\documentclass{amsart}

\usepackage{amsmath,amssymb,verbatim}
\usepackage{amsthm}
\usepackage{stmaryrd}
\usepackage{enumerate}
\usepackage{url}
\usepackage{wasysym}

\newtheorem{theorem}{Theorem}[section]
\newtheorem{fact}[theorem]{Fact}
\newtheorem{lemma}[theorem]{Lemma}
\newtheorem{corollary}[theorem]{Corollary}
\newtheorem{proposition}[theorem]{Proposition}

\theoremstyle{definition}
\newtheorem{definition}[theorem]{Definition}
\newtheorem{example}[theorem]{Example}
\newtheorem{remark}[theorem]{Remark}

\newcommand{\abar}{\bar{a}}
\newcommand{\bbar}{\bar{b}}
\newcommand{\cbar}{\bar{c}}
\newcommand{\dbar}{\bar{d}}
\newcommand{\mbar}{\bar{m}}
\newcommand{\nbar}{\bar{n}}
\newcommand{\ubar}{\bar{u}}
\newcommand{\xbar}{\bar{x}}
\newcommand{\ybar}{\bar{y}}
\newcommand{\zbar}{\bar{z}}
\newcommand{\claim}{\hfill$\dashv_{\text{\scriptsize{claim}}}$}

\def\eq{\operatorname{eq}}
\def\seq{\subseteq}
\def\nv{\text{-}}
\def\inv{^{\text{-}1}}
\def\Th{\operatorname{Th}}

\def\N{\mathbb{N}}
\def\R{\mathbb{R}}
\def\Z{\mathbb{Z}}
\def\Q{\mathbb{Q}}
\def\C{\mathbb{C}}
\def\K{\mathbb{K}}
\def\cL{\mathcal{L}}
\def\cA{\mathcal{A}}
\def\cB{\mathcal{B}}
\def\cG{\mathcal{G}}
\def\cZ{\mathcal{Z}}
\def\cK{\mathcal{K}}
\def\cM{\mathcal{M}}
\def\cN{\mathcal{N}}
\def\cU{\mathcal{U}}
\def\acl{\operatorname{acl}}
\def\qf{\operatorname{qf}}
\def\alg{\operatorname{alg}}
\def\tp{\operatorname{tp}}

\def\Ind{\setbox0=\hbox{$x$}\kern\wd0\hbox to 0pt{\hss$\mid$\hss}
\lower.9\ht0\hbox to 0pt{\hss$\smile$\hss}\kern\wd0}

\def\Notind{\setbox0=\hbox{$x$}\kern\wd0\hbox to 0pt{\mathchardef
\nn=12854\hss$\nn$\kern1.4\wd0\hss}\hbox to
0pt{\hss$\mid$\hss}\lower.9\ht0 \hbox to 0pt{\hss$\smile$\hss}\kern\wd0}

\def\ind{\mathop{\mathpalette\Ind{}}}

\title{Weakly minimal groups with a new predicate}

\date{November 25, 2019}

\author[G. Conant]{Gabriel Conant}

\address{Department of Mathematics\\
University of Notre Dame\\
Notre Dame, IN, 46656\\
 USA}
\email{gconant@nd.edu}

\author[M. C. Laskowski]{Michael C. Laskowski}

\address{Department of Mathematics\\
University of Maryland\\
College Park, MD 20742\\
 USA}
\email{mcl@math.umd.edu}

\begin{document}

\begin{abstract}
Fix a weakly minimal (i.e., superstable $U$-rank $1$) structure $\cM$. Let $\cM^*$ be an expansion by constants for an elementary substructure, and let $A$ be an arbitrary subset of the universe $M$. We show that all formulas in the expansion $(\cM^*,A)$ are equivalent to bounded formulas, and so $(\cM,A)$ is stable (or NIP) if and only if the $\cM$-induced structure $A_{\cM}$ on $A$ is stable (or NIP). We then restrict to the case that $\cM$ is a pure abelian group with a weakly minimal theory, and $A_{\cM}$ is mutually algebraic (equivalently, weakly minimal with trivial forking). This setting encompasses most of the recent research on stable expansions of $(\Z,+)$. Using various characterizations of mutual algebraicity, we give  new examples of stable structures of the form $(\cM,A)$. Most notably, we show that if $(G,+)$ is a weakly minimal additive subgroup of the algebraic numbers, $A\seq G$ is enumerated by a homogeneous linear recurrence relation with algebraic coefficients, and no repeated root of the characteristic polynomial of $A$ is a root of unity, then $(G,+,B)$ is superstable for any $B\seq A$.  
\end{abstract}

\maketitle

\vspace{-27pt}

\section{Introduction}

Given a structure $\cM$, and a set $A\seq M$, a common line of investigation concerns model-theoretic properties of $\cM$ that are preserved in the expansion $(\cM,A)$ of $\cM$ by a unary predicate naming $A$. In this situation, the \emph{$\cM$-induced structure on $A$}, denoted $A_{\cM}$ (see Definition \ref{def:induced}), is interpretable in $(\cM,A)$ and so model theoretic complexity in $A_{\cM}$ persists in $(\cM,A)$. Altogether, a fundamental question is when model theoretic tameness present in both $\cM$ and $A_{\cM}$ is also present in $(\cM,A)$. In \cite{CaZi}, Casanovas and Ziegler define the notion of a set $A\seq M$ that is \emph{bounded in $\cM$} (see Definition \ref{def:bounded}), which is a certain ``quantifier organization" property of formulas in the expansion $(\cM,A)$, and they show that if $A$ is bounded in $\cM$ then $(\cM,A)$ is stable if and only if $\cM$ and $A_{\cM}$ are stable. The analogous result for NIP was shown by Chernikov and Simon \cite{ChSiDP1}. 

A notable instance of the situation above concerns expansions of the complex field $(\C,+,\cdot)$ by a finite rank subgroup $\Gamma$ of a semi-abelian variety. In this setting, \emph{Lang's conjecture} (now a theorem of Faltings and Vojta) is equivalent to the statement  that $(\C,+,\cdot,\Gamma)$ is stable and $\Gamma_{(\C,+,\cdot)}$ is $1$-based. This equivalence is explained by Pillay in \cite{PiLC}, and also describes the model-theoretic ingredients of Hrushovski's \cite{HrMLFF} proof of Mordell-Lang for function fields. A consequence of Pillay's work is that if $\cM$ is strongly minimal, then any $A\seq M$ is bounded in $\cM$ (see \cite[Corollary 5.4]{CaZi}).

Drawing from results of Poizat \cite{Poizpairs} on ``beautiful pairs" of models of a stable theory, Casanovas and Ziegler \cite{CaZi} also isolate the more semantic notion of a \emph{small} set in $\cM$ (see Definition \ref{def:small}), and show that if $\cM$ is stable \emph{and} has nfcp, then small sets are bounded. Altogether, this yields a strategy for proving stability (or NIP)  of an expansion $(\cM,A)$ of an nfcp structure $\cM$: one first shows that $A$ is small in $\cM$ and then that $A_{\cM}$ is stable (or NIP). This strategy has been recently used by Palac\'{i}n and Sklinos \cite{PaSk}, the first author \cite{CoSS,CoGG}, and Lambotte and Point \cite{LaPo} to produce new families of stable expansions of $(\Z,+)$. 

The first main result of this paper is that if $\cM$ is \emph{weakly minimal} (i.e., superstable of $U$-rank $1$),  then \emph{any} subset of $M$ is bounded in the expansion of $\cM$ by constants for some $\cM_0\preceq\cM$ (see Theorem \ref{thm:wmt-bounded}). This generalizes Pillay's result on strongly minimal structures (modulo the extra constants for $\cM_0$, which are necessary; see Remark \ref{rem:assumptions}), and eliminates the role of smallness in the strategy for preserving stability used in \cite{CoSS}, \cite{CoGG}, \cite{LaPo}, and \cite{PaSk}. We also obtain the following conclusion about expansions of weakly minimal structures by unary predicates.

\newtheorem*{thm:wmt-main}{Theorem \ref{thm:wmt-main}}
\begin{thm:wmt-main}
Suppose $\cM$ is weakly minimal and $A\seq M$.
\begin{enumerate}[$(a)$]
\item If $A_{\cM}$ is stable of $U$-rank $\alpha$ then $(\cM,A)$ is stable of $U$-rank at most $\alpha\cdot\omega$. 
\item If $A_{\cM}$ is NIP then $(\cM,A)$ is NIP.
\end{enumerate}
\end{thm:wmt-main}

In the rest of the paper, we focus on expansions of pure abelian groups $\cG=(G,+)$ whose theory is weakly minimal (see Proposition \ref{prop:wmagchar} for an algebraic characterization of such groups). Given $A\seq G$,  the induced structure $A_{\cG}$ decouples into two parts, namely, the \emph{quantifier-free} induced structure $A^{\qf}_{\cG}$ (see Definition \ref{def:indqf}), along with unary predicates $A\cap nG$ for all $n\geq 1$. The focus of \cite{LaPo} and \cite{PaSk}  is on expansions of $\cZ=(\Z,+)$ by sets $A\seq\Z$ that are eventually periodic modulo any fixed $n\geq 1$, which implies good behavior for these extra unary predicates.  However, as observed by the first author in \cite{CoSS,CoGG},  the sets $A\seq\Z$ considered in \cite{CoSS}, \cite{CoGG}, \cite{LaPo}, and \cite{PaSk} all have the property that \emph{any} expansion of $A^{\qf}_{\cZ}$ by unary predicates is stable, and so this extra assumption of periodicity is unnecessary.  In the present paper, we isolate a model-theoretic setting for this phenomenon. Specifically, we consider  \emph{mutually algebraic structures}, which were defined by the second author in \cite{LaskMAS}, and shown to satisfy the property that any expansion by unary predicates is stable and has nfcp. In Section \ref{sec:WMAG}, we prove the following result.

\newtheorem*{thm:MAG-main}{Theorem \ref{thm:MAG-main}}
\begin{thm:MAG-main}
Let $\cG=(G,+)$ be a weakly minimal abelian group. 
Fix $A\seq G$, and suppose $A_{\cG}^{\qf}$ is mutually algebraic.  Then, for any finite $F\subset G$ and any $B\seq A+F$, $(\cG,B)$ is  superstable of $U$-rank at most $\omega$. Moreover, if $\cG$ has finite torsion then $(\cG,B)$ has nfcp; and if $\cG=(\Z,+)$ and $B$ is infinite then $(\cG,B)$ has $U$-rank $\omega$.  
\end{thm:MAG-main}

Preservation of nfcp in Theorem \ref{thm:MAG-main}, is obtained via a result of Casanovas and Ziegler \cite{CaZi} involving small sets (see Proposition \ref{prop:CZsmall}). In order to apply \cite{CaZi}, we show  that if $\cG=(G,+)$ is a weakly minimal abelian group with finite torsion, and $A\seq G$ is such that $A_{\cG}$ is mutually algebraic, then $A$ is small in $\cG$ (see Corollary \ref{cor:interpretG}). This result is also interesting in light of the fact that smallness of $A$ and mutual algebraicity of $A_{(\Z,+)}$ are given as separate arguments in the work on stable expansions of the form $(\Z,+,A)$ from \cite{CoSS}, \cite{CoGG}, \cite{LaPo}, and \cite{PaSk}.

Finally, in Section \ref{sec:Zexp}, we use Theorem \ref{thm:MAG-main} to find new examples of stable expansions of weakly minimal abelian groups. In particular, we show that if $\cG=(G,+)$ is a weakly minimal abelian group, $A$ is a subset of $G$, and one of the following situations holds, then $A_{\cG}^{\qf}$ is mutually algebraic and so   Theorem \ref{thm:MAG-main} applies.

\begin{enumerate}
\item [\textbf{\S\ref{sec:Kepler}.}] 
$\cG$ is a subgroup of $(\C,+)$, $A=\{a_n\}_{n=0}^\infty$, and either  $\lim_{n\to\infty}|\frac{a_{n+1}}{a_n}|=\infty$ or $\lim_{n\to\infty}\frac{a_{n+1}}{a_n}=\tau$ for some transcendental $\tau\in\C$ with $|\tau|>1$.

\item [\textbf{\S\ref{sec:FGM}.}] 
$\cG$ is an additive subgroup of an algebraically closed field $\K$ of characteristic $0$, and $A$ is contained in a finite rank multiplicative subgroup of $\K^*$.

\item [\textbf{\S\ref{sec:LRR}.}] 
$\cG$ is a subgroup of $(\Q^{\alg},+)$, and $A$ is enumerated by a linear homogeneous recurrence relation with constant coefficients whose characteristic polynomial  has no repeated root that is also a root of unity.
\end{enumerate}

The examples in Section \ref{sec:Kepler} generalize certain families of ``sparse sets" considered in \cite{CoSS}, \cite{LaPo}, and \cite{PaSk}. In this case, we use methods similar to Lambotte and Point \cite{LaPo} to show that $A^{\qf}_{\cG}$ is interdefinable with $A$ in the language of equality. 

The examples in Section \ref{sec:FGM} generalize work of the first author from \cite{CoGG}, and complement many existing results about expansions of the \emph{field} $(\C,+,\cdot)$ by finite rank multiplicative subgroups (e.g., Belegradek \& Zilber \cite{BeZi}, and Van den Dries \& G\"{u}nayd{\i}n \cite{vdDGu}). In this case, we use a number-theoretic result of Everste, Schlickewei, and Schmidt \cite{ESS} to give an extremely quick proof that $A^{\qf}_{\cG}$ is mutually algebraic. This proof also highlights a pleasing parallel between the original definition of mutual algebraicity from \cite{LaskMAS} and the behavior of solutions of linear equations which lie in a finite rank multiplicative group.

Section \ref{sec:LRR} provides a significant generalization of the examples from \cite{CoSS} and \cite{LaPo} of stable structures of the form $(\Z,+,A)$, where $A$ is enumerated by a homogeneous linear recurrence relation. These sources impose fairly restrictive assumptions such as irreducibility of the characteristic polynomial $p(x)$ of the recurrence. In particular, stability of $(\Z,+,A)$ even when $p(x)$ is separable was open. Here we consider a weakly minimal subgroup $\cG$ of $(\Q^{\alg},+)$, and a set $A\seq G$ enumerated by a recurrence relation whose characteristic polynomial $p(x)$ has no roots of unity as repeated roots. To prove  $A^{\qf}_{\cG}$ is mutually algebraic, we first exhibit mutual algebraicity of an auxiliary structure $\N^{\Phi}_{\cK}$, formulated using a number field over which $A$ is defined, and then show that $A^{\qf}_{\cG}$ is suitably interpreted in $\N^{\Phi}_{\cK}$. The proof of mutual algebraicity for $\N^{\Phi}_{\cK}$ uses a new characterization of mutual algebraic structures, due to the second author and Terry \cite{LasTe1}, as well as a quantitative version of work of M. Laurent \cite{MLaur1,MLaur2}, due to Schlickwei and Schmidt \cite{SchSchPEE}, on the number of solutions to polynomial-exponential equations over number fields. Finally, in Theorem \ref{thm:LRRC}, we give a more direct proof in the case that $p(x)$ is separable, which also applies to any algebraically closed field of characteristic $0$.

\subsection*{Acknowledgements} This work began while both authors were participants in the Model Theory, Combinatorics and Valued Fields trimester program at Insitut Henri Poincar\'{e} in Spring 2018. We thank IHP for their hospitality. We also thank Tom Scanlon for suggesting the work of M. Laurent, and the anonymous referee for valuable revisions and suggestions. The second author was partially supported by NSF grant DMS-1308546. 

\section{Bounded sets in weakly minimal theories}\label{sec:WMG}

 Throughout this section, let $T$ be a complete theory with infinite models in some language $\cL$.  Given $\cM\models T$, when we say that a set $X\seq M^n$ is \emph{$\cM$-definable}, we mean definable with parameters from $M$. Let $\cL(P)=\cL\cup\{P\}$ where $P$ is a unary relation symbol not in $\cL$. Given $\cM\models T$ and $A\seq M$, let $(\cM,A)$ be the $\cL(P)$-structure expanding $\cM$ in which $P$ is interpreted as $A$. 
 
 \begin{definition}\label{def:bounded}$~$
 \begin{enumerate}
\item  An $\cL(P)$-formula $\phi(x_1,\ldots,x_n)$ is \textbf{bounded} if it is of the form
 \[
 Q_1 y_1\in P\ldots Q_m y_m\in P\, \psi(x_1,\ldots,x_n,y_1,\ldots,y_m)
 \]
 for some quantifiers $Q_1,\ldots,Q_m$ and some $\cL$-formula $\psi(\xbar,\ybar)$.
 
\item Given $\cM\models T$,  a set $A\seq M$ is \textbf{bounded in $\cM$} if every $\cL(P)$-formula is equivalent, modulo $\Th(\cM,A)$, to a bounded $\cL(P)$-formula. 
\end{enumerate}
 \end{definition}

\begin{definition}\label{def:induced} 
Given $\cM\models T$ and a sort $S$ from $\cL$, let $\cL^{\cM}_S$ denote the relational language containing an $n$-ary relation $R_X$ of sort $S^n$, for any $n\geq 1$ and any $\cM$-definable $X\seq (M^S)^n$. Given $A\seq M^S$, let $A_{\cM}$ denote the $\cL_S^{\cM}$-structure, with universe $A$, in which each symbol $R_X$ is interpreted as $A^n\cap X$. We call $A_{\cM}$ the \textbf{$\cM$-induced structure on $A$}. 
\end{definition}

The following is Proposition 3.1 of \cite{CaZi}. 
  
 \begin{proposition}[Casanovas \& Ziegler]\label{prop:CZ-bounded}
 Fix $\cM\models T$ and $A\seq M$. If $A$ is bounded in $\cM$, then $(\cM,A)$ is stable if and only if $\cM$ and $A_{\cM}$ are stable. 
\end{proposition}

We will also use the following characterization of bounded sets in stable theories, which is part of Proposition 5.3 of \cite{CaZi}. 

\begin{proposition}[Casanovas \& Ziegler]\label{prop:5.3}
If $T$ is stable then, for any $\cM\models T$ and $A\seq M$, the following are equivalent.
\begin{enumerate}[$(i)$]
\item $A$ is bounded in $\cM$.
\item If $(\cN,B)\equiv_{\cL(P)}(\cM,A)$ is $|T|^+$-saturated, $f$ is an $\cL$-elementary map in $\cN$, which is a finite extension of a permutation of $B$, and $a\in N$, then there is $b\in N$ such that $f\cup\{(a,b)\}$ is $\cL$-elementary.
\end{enumerate}
\end{proposition}

For the rest of the paper, we will focus on expansions of weakly minimal theories.

\begin{definition}
A complete theory $T$ with infinite models is \textbf{weakly minimal} if it is stable and any forking extension of a $1$-type is algebraic.
\end{definition}

We will also call models of $T$ \textbf{weakly minimal} when $T$ is weakly minimal.   Recall that any stable theory has a $U$-rank in $\text{Ord}\cup\{\infty\}$, which is ordinal-valued if and only if the theory is superstable. Recall also that $T$ is weakly minimal if and only if it is stable of $U$-rank $1$. The following result  is  \cite[Theorem 2.11]{CoSS}, and is proved using Proposition \ref{prop:CZ-bounded} and techniques similar to the work of Palac\'{i}n and Sklinos \cite{PaSk} on the expansion of $(\Z,+)$ by $\{2^n:n\in\N\}$.  (We use $\cdot$ for multiplication of ordinals, which extends to $\text{Ord}\cup\{\infty\}$ in the obvious way.)

\begin{theorem}[Conant]\label{thm:CoSS}
Assume $T$ is weakly minimal and fix $\cM\models T$. Suppose $A\seq M$ is bounded in $\cM$ and is such that $A_{\cM}$ is stable of $U$-rank $\alpha$. Then $(\cM,A)$ is stable of $U$-rank at most $\alpha\cdot\omega$.
\end{theorem}

\begin{definition}$~$
\begin{enumerate}
\item Given $\cM\models T$, let $\cL_{M}$ be the expansion of $\cL$ by adding a constant symbol for each element of $M$, and let $T_{\cM}$ be the elementary diagram of $\cM$ in the expanded language $\cL_{M}$.
\item Fix $\cM\models T$ and $\cM_0\preceq\cM$. A set $A\seq M$ is \textbf{bounded in $\cM$ with respect to $\cL_{M_0}$} if it is bounded in the canonical $\cL_{M_0}$-expansion of $\cM$, i.e., for all $\cL(P)$-formulas $\phi(\xbar;\ybar)$ and all $\bbar\in M_0^{\ybar}$, there is a bounded $\cL(P)$-formula $\psi(\xbar;\zbar)$ and $\cbar\in M_0^{\zbar}$ such that $(\cM,A)\models\forall \xbar(\phi(\xbar;\bbar)\leftrightarrow\psi(\xbar;\cbar))$. 
\end{enumerate}
\end{definition}

We now state our first main result.

\begin{theorem}\label{thm:wmt-bounded}
If $T$ is weakly minimal, $\cM\models T$, and $\cM_0\preceq\cM$, then every subset of $M$ is bounded in $\cM$ with respect to $\cL_{M_0}$.
\end{theorem}

Before continuing to the proof, we use Theorem \ref{thm:wmt-bounded} to establish the second main result of this section.

\begin{theorem}\label{thm:wmt-main}
Suppose $T$ is weakly minimal, $\cM\models T$, and $A\seq M$.
\begin{enumerate}[$(a)$]
\item If $A_{\cM}$ is stable of $U$-rank $\alpha$, then $(\cM,A)$ is stable of $U$-rank at most $\alpha\cdot\omega$.
\item If $A_{\cM}$ is NIP then $(\cM,A)$ is NIP.
\end{enumerate}
\end{theorem}
\begin{proof}
By definition of $A_{\cM}$, we may assume without loss of generality that $\cL=\cL_{M}$ and $T=T_{\cM}$. By Theorem \ref{thm:wmt-bounded}, $A$ is bounded in $\cM$. So part $(a)$ follows from Theorem \ref{thm:CoSS}, and part $(b)$ follows from \cite[Corollary 2.5]{ChSiDP1}.
\end{proof}

The proof of Theorem \ref{thm:wmt-bounded} breaks into several pieces. We first note various facts about weakly minimal theories. First, note that if $T$ is weakly minimal and $\cM\models T$, then $T_{\cM}$ is weakly minimal. 

\begin{lemma}\label{lem:wmfacts}
Suppose $T$ is weakly minimal, $\cM_0\models T$, $\cM_0\preceq\cM$, and $M_0\seq A\seq M$. Then $\acl(A)\models T$ and $\cM_0\preceq\acl(A)\preceq\cM$. Moreover, if $\cM_0$ is $|T|^+$-saturated, then  $\acl(A)$ is $|T|^+$-saturated as well.    
\end{lemma}
\begin{proof}
Without loss of generality, assume $A=\acl(A)$. To show $A\preceq \cM$, we choose an $\cL$-formula $\phi(x;\abar)$, with $\abar\subset A$, that has a solution $b\in M$, and we show that $\phi(x;\abar)$ has a solution in $A$. If $b\in A$ we are done, so assume $b\not\in A$. Then $b\ind_{M_0}\abar$ and so, by finitely satisfiability, there is $b^*\in M_0$ such that $\cM\models\varphi(b^*;\abar)$, as desired. Next, since $\cM_0\preceq \cM$ and $M_0\seq A\seq M$, it follows that $\cM_0\preceq A$.

Now assume $\cM_0$ is $|T|^+$-saturated. We argue that any model $\cN\succeq\cM_0$ must also be $|T|^+$-saturated, which suffices. The proof is essentially the same as \cite[Proposition 3.5]{GooLask} (in fact, the following argument can be adapted to any non-multidimensional theory by replacing the use of weak minimality with an appropriate version of the ``three-model lemma"). Let $\cN^*$ be the $|T|^+$-prime model over $N$. If $\cN=\cN^*$ we finish, so assume otherwise. Choose $b\in N^*\backslash N$. Then $\tp(b/N)$ is a non-algebraic extension of $\tp(b/M_0)$ and so $b\ind_{M_0} N$ by weak minimality. Since $N^*$ is dominated by $N$ over $M_0$ (adapt \cite[Lemma 1.4.3.4$(iii)$]{Pibook} to the category of $|T|^+$-saturated models), we have $b\ind_{M_0}N^*$, which is a contradiction. 
\end{proof}

Suppose now that $T$ is weakly minimal. Then a $1$-type over a model of $T$ is regular if and only if it is non-algebraic. Suppose $\cM\preceq\cN$ are $|T|^+$-saturated models of $T$. By weak minimality and exchange for algebraic independence, we have that for any regular  $p,q\in S_1(M)$, if $p$ and $q$ are non-orthogonal, and $I\seq p(N)$ and $J\seq q(N)$ are maximal $M$-independent sets, then $|I|=|J|$ (note that by $|T|^+$ saturation, orthogonality and weak orthogonality coincide for regular types over $M$; see \cite[Lemma 1.4.3.1]{Pibook}). So, for any regular type $p$ over some other model of $T$, we have a well-defined dimension $\dim_p(N/M)$, namely, the cardinality of a maximal $M$-independent set of realizations in $N$ of any regular $q\in S_1(M)$ non-orthogonal to $p$.  The following property of $\dim_p$ is a standard exercise (see \cite{Pibook}, \cite{Shbook}).

\begin{fact}\label{fact:weight}
Assume $T$ is weakly minimal and $\cM\preceq\cN\models T$ are $|T|^+$-saturated. Suppose $p$ is a regular type over some model of $T$, and $\cN^*\succeq\cN$. Then $\dim_p(N^*/M)$ is finite if and only if $\dim_p(N^*/N)$ and $\dim_p(N/M)$ are finite, and in this case $\dim_p(N^*/M)=\dim_p(N^*/N)+\dim_p(N/M)$.
\end{fact}

We now prove a proposition that carries additional hypotheses, which we subsequently remove in the proof of Theorem \ref{thm:wmt-bounded}. 

\begin{proposition}\label{prop:sat-version}
Suppose $T$ is weakly minimal, $\cM_0\models T$ is $|T|^+$-saturated, $\cM_0\preceq\cM$, and $A\seq M$ is $\cL_{M_0}$-algebraically closed (so, in particular, $M_0\seq A\seq M$). Then $A$ is bounded in $\cM$ with respect to $\cL_{M_0}$.
\end{proposition}
\begin{proof}
We will apply Proposition \ref{prop:5.3} with respect to the $\cL_{M_0}$-theory $T_{\cM_0}$. Given $A$ as in the statement, choose any sufficiently saturated  $(\cM^*,A^*)\succeq_{\cL_{M_0}(P)}(\cM,A)$. Fix finite $\bbar,\cbar\subset M^*$ and an $\cL_{M_0}$-elementary bijection $f\colon A^*\bbar\to A^*\cbar$ extending a permutation of $A^*$. Choose any $d\in M^*\backslash A^*\bbar$. It suffices to find $d'\in M^*$ such that $f\cup\{(d,d')\}$ is $\cL_{M_0}$-elementary. By Lemma \ref{lem:wmfacts}, the structures $A^*$, $\cM_1:=\acl_{\cL_{M_0}}(A^*\bbar)$, and $\cM_2:=\acl_{\cL_{M_0}}(A^*\cbar)$ are $|T|^+$-saturated models of $T_{\cM_0}$. Note that we can extend $f$ to an $\cL_{M_0}$-isomorphism $f^*\colon \cM_1\to\cM_2$. Let $p:=\tp(d/M_1)$ and $p':=f^*(p)\in S_1(M_2)$. We want to show that $p'$ is realized by some $d'\in M^*$. 
 
 Now, if $d\in M_1$ then we are done, so assume otherwise. Then $p$ and $p'$ are regular, and have the same restriction to $M_0$ since $f^*$ is $\cL_{M_0}$-elementary. In particular, $p$ and $p'$ are non-orthogonal, and so $\dim_{p'}(M^*/M_2)=\dim_p(M^*/M_2)$. So, to show $p'$ is realized in $\cM^*$, it suffices to show $\dim_{p}(M^*/M_2)>0$. For a contradiction, suppose $\dim_p(M^*/M_2)=0$. Note that $\dim_p(M_2/A_*)$ is finite since it is bounded above by $\dim_{\acl}(M_2/A^*)$. So  $\dim_p(M^*/A^*)=\dim_p(M_2/A^*)$ by Fact \ref{fact:weight}. Note also that $\dim_{p}(M_1/A^*)=\dim_p(M_2/A^*)$ since $f^*$ is an $\cL_{M_0}$-isomorphism extending a permutation of $A^*$. So $\dim_p(M^*/A^*)=\dim_p(M_1/A^*)$ which, by Fact \ref{fact:weight}, yields $\dim_p(M^*/M_1)=0$. But this contradicts our assumption that $d\not\in M_1$.
\end{proof}

We can now prove Theorem \ref{thm:wmt-bounded}.

\begin{proof}[Proof of Theorem \ref{thm:wmt-bounded}]
Assume $T$ is weakly minimal, $\cM\models T$, and $\cM_0\preceq\cM$. Choose $A\seq M$ arbitrarily. We want to show $A$ is bounded in $\cM$ with respect to $\cL_{M_0}$. Consider the $\cL(P,Q)$-structure $(\cM,A,M_0)$. Choose a $|T|^+$-saturated $\cL(P,Q)$-elementary extension $(\cM^*,A^*,M^*_0)$, and note that $M^*_0$ is the universe of a $|T|^+$-saturated $\cL$-elementary extension $\cM^*_0$ of $\cM_0$.

We now work with the theory $T_{\cM^*_0}$ in the language $\cL^*:=\cL_{M^*_0}$. Let $(\cN^*,B)\equiv_{\cL^*(P)}(\cM^*,A^*)$ be $|T_{\cM^*_0}|^+$-saturated. Let $B^*=\acl_{\cL^*}(B)$. We have that $T$ is weakly minimal, $\cM^*_0\models T$ is $|T|^+$-saturated, $\cM^*_0\preceq\cN^*$, and $B^*\seq N^*$ is $\cL^*$-algebraically closed. So we may apply Proposition \ref{prop:sat-version} to conclude that $B^*$ is bounded in $\cN^*$ with respect to $\cL^*$. Now, suppose $\cbar,\dbar\subset N^*$ are finite and $h\colon B\cbar\to B\dbar$ is an $\cL^*$-elementary bijection in $\cN^*$ extending a permutation of $B$. Then $h$ extends to an $\cL^*$-elementary bijection $h^*\colon B^*\cbar\to B^*\dbar$. Since $B^*$ is bounded in $\cN^*$ with respect to $\cL^*$, Proposition \ref{prop:5.3} implies that for every $a\in N^*$ there is an $a'\in N^*$ such that $h^*\cup\{(a,a')\}$ is $\cL^*$-elementary in $\cN^*$. Applying Proposition \ref{prop:5.3} again, we conclude that $B$ is bounded in $\cN^*$ with respect to $\cL^*$. By elementarity, $A^*$ is bounded in $\cM^*$ with respect to $\cL^*$.  

Now, fix any $\cL(P)$-formula $\phi(\xbar;\ybar)$ and let $\Gamma(\ybar)$ be the $\cL(P,Q)$-type 
\[
\{\ybar\in Q\}\cup\{\forall \zbar\in Q\,\neg\forall\xbar (\phi(\xbar;\ybar)\leftrightarrow \psi(\xbar;\zbar)):\psi(\xbar;\zbar)\text{ is a bounded $\cL(P)$-formula}\}.
\]
Since $A^*$ is bounded in $\cM^*$ with respect to $\cL^*$, $\Gamma(\ybar)$ is not realized in $\cN:=(\cM^*,A^*,M^*_0)$. By saturation of $\cN$, $\Gamma(\ybar)$ is inconsistent with $\Th(\cN)$. By compactness, there are finitely many bounded $\cL(P)$-formulas $\psi_1(\xbar;\zbar_1),\ldots,\psi_\ell(\xbar;\zbar_\ell)$ such that
\[
\cN\models\forall\ybar\in Q\bigvee_{i=1}^{\ell}\exists \zbar_i\in Q\,\forall\xbar(\phi(\xbar;\ybar)\leftrightarrow\psi_i(\xbar;\zbar_i)).
\]
So $(\cM,A,M_0)$ realizes this sentence, and so we see that for every $\abar\in M_0^{\ybar}$ there is $1\leq i\leq\ell$ and $\cbar\in M_0^{\zbar_i}$ such that $(\cM,A)\models\forall \xbar(\phi(\xbar;\bbar)\leftrightarrow\psi_i(\xbar;\cbar))$. 

As the $\cL(P)$-formula $\phi(\xbar;\ybar)$ above was arbitrary, we conclude that $A$ is bounded in $\cM$ with respect to $\cL_{M_0}$. 
\end{proof}

\begin{remark}\label{rem:assumptions}
We make some comments on the assumptions in Theorem \ref{thm:wmt-bounded}.
\begin{enumerate}[$(1)$]
\item  Theorem \ref{thm:wmt-bounded} cannot be generalized to arbitrary stable theories. For example, Poizat \cite{Poizpairs} constructed an $\omega$-stable theory $T$ and $\cN\prec\cM\models T$ such that the pair $(\cM,N)$ is unstable. By stability of $T$, $N_{\cM}$ is the same as $\cN$  and so, by Proposition \ref{prop:CZ-bounded}, $N$ is not bounded in $\cM$ (or any expansion by constants). In \cite{Bouspairs}, Bouscaren shows that if $T$ is superstable, then every theory of pairs of models of $T$ is stable if and only if $T$ does not have the dimensional order property.

\item The additional constants naming a substructure are necessary in order to prove Theorem \ref{thm:wmt-bounded}. For example,  let $T$ be the theory of an equivalence relation $E$ with two infinite classes. Fix $\cM\models T$ and distinct $a_1,a_2,b\in M$ such that $E(a_1,a_2)$ and $\neg E(a_1,b)$. Then $A=M\backslash\{a_1,a_2,b\}$ is not bounded in $\cM$. To see this, note that $a_1$ and $b$ clearly have different $\cL(P)$-types while, on the other hand, there is an $\cL$-elementary map from $Aa_1$ to $Ab$, extending a permutation of $A$, and so $a_1$ and $b$ satisfy the same bounded $\cL(P)$-formulas.
\end{enumerate}
\end{remark}

In \cite{PiLC}, Pillay proves that if $T$ is \emph{strongly minimal}, $\cM\models T$, and $A\seq M$, then $A$ is bounded in $\cM$ without adding any extra constants (see also \cite[Corollary 5.4]{CaZi}).  
Although it will not be necessary for our later results, it is interesting to see that the same holds for weakly minimal \emph{expansions of groups}. 

\begin{theorem}\label{thm:noconstants}
Suppose $T$ is the theory of a weakly minimal expansion of a group, and $\cM\models T$. Then every subset of $M$ is bounded in $\cM$.
\end{theorem}
\begin{proof}
Fix $A\seq M$ and let $(\cG,B)\succeq_{\cL(P)}(\cM,A)$ be $|T|^+$-saturated. Fix finite $\cbar,\dbar\subset G$ and suppose $f\colon B\cbar\to B\dbar$ is a partial $\cL$-elementary map extending a permutation of $B$. Fix $a\in G$. We want to find  $b\in G$ such that $f\cup\{(a,b)\}$ is $\cL$-elementary. Then $A$ will be bounded in $\cM$ by Proposition \ref{prop:5.3}. 

For the rest of the proof, we work in $T$. We will focus on $1$-types over $\acl^{\eq}(\emptyset)$, for which we have the following simplified version of the dimension used above. Given  $p\in S_1(\acl^{\eq}(\emptyset))$ and algebraically closed sets $C\seq D\seq G$, let $\dim_p(D/C)$ be the cardinality of a maximal $C$-independent subset of $p(D)$ (which is well-defined by weak minimality). 

\medskip

\noindent\emph{Claim:} Fix $p,q\in S_1(\acl^{\eq}(\emptyset))$ and algebraically closed sets $C\seq D\seq G$.
\begin{enumerate}[$(a)$]
\item $\dim_p(G/C)$ is finite if and only if $\dim_p(G/D)$ and $\dim_p(D/C)$ are finite, and in this case $\dim_p(G/C)=\dim_p(G/D)+\dim_p(D/C)$.
\item If  $p$ and $q$ are both realized in $C$, then $\dim_p(G/C)=\dim_q(G/C)$.
\end{enumerate}

\noindent\emph{Proof}: 
Part $(a)$ follows from the fact that if $I\seq p(G)$ is a maximal $D$-independent set and $J\seq p(D)$ is a maximal $C$-independent set, then $I\cap J=\emptyset$ and, by exchange, $I\cup J$ is a maximal $C$-independent subset of $p(G)$.

Part $(b)$. It suffices to show $\dim_p(G/C)\leq\dim_q(G/C)$. Fix $b_0,c_0\in C$ realizing $p$ and $q$, respectively. Given any $C$-independent set $I\seq p(G)$, let $J=\{ab\inv_0 c_0:a\in I\}$. Then we clearly have that $J$ is $C$-independent, and that $|J|=|I|$. Moreover, for any $a\in I$, we have $\text{stp}(a)=\text{stp}(b_0)$, and so $ab_0\inv\in G^0=\text{Stab}(q)$, which implies $ab\inv_0c_0\models q$. So $J\seq q(G)$ and, altogether, $\dim_p(G/C)\leq \dim_q(G/C)$.\claim 

\medskip

Now let $C_1=\acl(B\cbar)$ and $C_2=\acl(B\dbar)$.  Without loss of generality, we may extend $f$ so that it is a map from $C_1$ to $C_2$.  Let $p=\text{stp}(a)$, and let $\cG^*$ be a sufficiently saturated elementary extension of $\cG$. Choose $b_*\in G^*$ such that $f\cup \{(a,b_*)\}$ is elementary, and let $q=\text{stp}(b_*)$. If $b_*\in G$ then we are done, so assume otherwise. In particular, $b_*\not\in C_2$, which implies $a\not\in C_1$ and $b_*\ind_\emptyset C_2$. To find our desired $b$, it suffices by the stationarity of $q$ to find $b\in G\backslash C_2$ realizing $q$. In other words, we want to show $\dim_q(G/C_2)>0$.

Suppose first that $p$ is not realized in $C_1$. Since $\cG$ is $|T|^+$-saturated, there is a realization $b$ of $q$ in $G$. Toward a contradiction, suppose $b\in C_2$. Then $\text{stp}(b_*)=\text{stp}(b)$, and so $b_*\inv b\in (G^*)^0$. Then $a\inv f\inv(b)\in G^0$, and so $\text{stp}(a)=\text{stp}(f\inv(b))$, which contradicts that $p$ is not realized in $C_1$.

Next, let $r\in S_1(\acl^{\eq}(\emptyset))$ be the principal generic. Suppose $r$ is not realized in $C_1$. Since $r$ is $\emptyset$-invariant, it is also not realized in $C_2$. Note that if $b_1,b_2\models q$, with $b_1\ind_\emptyset b_2$, then $b_1\inv b_2\models r$. So we have $\dim_q(C_2/\emptyset)\leq 1$. Since $\dim_q(G/\emptyset)$ is infinite (by $|T|^+$-saturation of $\cG$), it follows that $\dim_q(G/C_2)$ is infinite.  

Finally, suppose $p$ and $r$ are both realized in $C_1$. As above, $r$ is realized in $C_2$. Also $q$ is realized in $C_2$ since $f(p(C_1))\seq q(C_2)$. By part $(b)$ of the claim, 
\[
\dim_p(G/C_1)=\dim_r(G/C_1)\makebox[.5in]{and}\dim_q(G/C_2)=\dim_r(G/C_2).
\]
In particular, we may assume $\dim_r(G/C_2)$ is finite. Note also that $\dim_r(C_2/B)$ is finite since it is bounded above by $\dim_{\acl}(C_2/B)$. By part $(a)$ of the claim,
\[
\dim_r(G/C_1)+\dim_r(C_1/B)=\dim_r(G/B)=\dim_r(G/C_2)+\dim_r(C_2/B).
\]
Since $f\colon C_1\to C_2$ extends a permutation of $B$, and $r$ is $\emptyset$-invariant, we also have $\dim_r(C_1/B)=\dim_r(C_2/B)$, and so $\dim_r(G/C_1)=\dim_r(G/C_2)$. Altogether, this yields $\dim_p(G/C_1)=\dim_q(G/C_2)$. Since $\dim_p(G/C_1)>0$ (witnessed by $a$), we have $\dim_q(G/C_2)>0$.  
\end{proof}

\section{Expansions by unary predicates}\label{sec:WMAG}

In this section, we give our main general result on preservation of stability and nfcp in expansions of  weakly minimal abelian groups  (see Theorem \ref{thm:MAG-main} below). Along with the work from Section \ref{sec:WMG}, the proof will require several more ingredients.

\subsection{Weakly minimal abelian groups}

By a \emph{pure} group, we mean a group  as a structure in the group language. It is well-known that any pure abelian group is stable and eliminates quantifiers in the expansion by binary relations for equivalence modulo $n$ for all $n\geq 1$ (see, e.g., \cite{IbKiTa}). By a \textbf{weakly minimal abelian group}, we mean an infinite abelian group $(G,+)$ whose pure theory is weakly minimal. The following folklore result gives an algebraic characterization of such groups.

\begin{proposition}\label{prop:wmagchar}
An infinite abelian group $(G,+)$ is weakly minimal if and only if, for all $n\geq 1$, the subgroups  $nG=\{nx:x\in G\}$ and $\textnormal{Tor}_n(G)=\{x\in G:nx=0\}$ are each either finite or of finite index.
\end{proposition}
\begin{proof}
It is well-known that a pure abelian group is weakly minimal if and only if every definable subgroup is  finite or of finite index (see, e.g., \cite[Proposition 2.1]{HruLov}). So suppose we only know that this holds for the subgroups $nG$ and $\textnormal{Tor}_n(G)$. Given $\cM\succeq (G,+)$, any definable subset of $M$ is a finite Boolean combination of sets defined by $mx=r$ or $mx\equiv_n r$ where $m,n\geq 1$ and $r\in M$. Using the assumption on $G$, it is straightforward to check that any such set is finite or $\cL_G$-definable. It follows that $(G,+)$ is weakly minimal (see \cite[Theorem 21]{BPW}).
\end{proof}

We will say that $(G,+)$ has \textbf{finite torsion} if $\textnormal{Tor}_n(G)$ is finite for all $n\geq 1$.

\begin{proposition}\label{prop:wmagacl}
If $(G,+)$ is a weakly minimal abelian group with finite torsion and $A\seq G$, then  $\acl(A)$ is the divisible hull of the subgroup generated by $A$.
\end{proposition}
\begin{proof}
Fix $A\seq G$ and let $H$ be the divisible hull of $\langle A\rangle $. Then $H\seq \acl(A)$ since $G$ has finite torsion. So suppose $\phi(x)$ is an algebraic formula with parameters from $A$. By quantifier elimination, we may assume $\phi(x)$ is a conjunction of  formulas  of the form $mx=r$, $mx\equiv_n r$, and their negations, where $m,n\geq 1$ and $r\in \langle A\rangle$. Since $G$ has finite torsion, any formula of the form $mx=r$ has finitely many solutions, which all lie in $H$. Moreover, if $n\geq 1$ then $nG\cong G/\textnormal{Tor}_n(G)$ is infinite, and thus has finite index. So the negation of a formula of the form $mx\equiv_n r$ is equivalent to a finite disjunction of positive instances of such formulas. Altogether, we may assume $\phi(x)$ is a conjunction of formulas of the form $mx\equiv_n r$. As $nG$ is infinite for all $n\geq 1$, such a formula either has infinitely many solutions or no solutions.
\end{proof}

\begin{definition}\label{def:indqf}$~$
\begin{enumerate}
\item Given an $\cL$-structure $\cM$, and a set $A\seq M$, let $A^{\qf}_{\cM}$ denote the reduct of $A_{\cM}$ to relations of the form $A^n\cap X$, for any $n\geq 1 $ and $X\seq M^n$  definable by a quantifier-free $\cL_M$-formula. 
\item We say that two structures $\cM_1$ and $\cM_2$, with the same universe $M$ (but possibly different languages), are \textbf{interdefinable} if, for any $n\geq 1$ and $X\seq M^n$, $X$ is $\cM_1$-definable if and only if it is $\cM_2$-definable.
\end{enumerate}
\end{definition}

The following is an immediate consequence of quantifier elimination. 

\begin{proposition}\label{prop:wmagind}
If $\cG=(G,+)$ is an abelian group and $A\seq G$ then $A_{\cG}$ is interdefinable with the expansion of $A_{\cG}^{\qf}$ by unary predicates $A\cap nG$ for $n\geq 1$.
\end{proposition}

\subsection{Mutually algebraic structures}
In the examples of stable expansions of $\cG=(\Z,+)$ from \cite{CoSS}, \cite{CoGG}, \cite{LaPo}, and \cite{PaSk}, the sets $A\seq\Z$ of interest turn out to satisfy the property that that \emph{any} expansion of $A^{\qf}_{\cG}$ by unary predicates is stable. In this section, we describe a robust model-theoretic setting for this behavior, which will be crucial for later results.  Let $\cA$ denote an arbitrary $\cL$-structure with universe $A$.

\begin{definition}\label{def:MA}
Fix an $\cL_A$-formula $R(z_1,\ldots,z_n)$. We say that $R$ is \textbf{mutually algebraic} if there is an integer $N\geq 1$ such that, for any $1\leq i\leq n$ and any $b\in A$, 
\[
\left|\left\{(a_1,\ldots,a_{n-1})\in A^{n-1}:\cA\models R(a_1,\ldots,a_{i-1},b,a_i,\ldots,a_{n-1})\right\}\right|\leq N.
\]
(In particular, any unary formula is mutually algebraic.)

Let $\cL_R$ be the language containing just the relation $R(\zbar)$. Given a nonempty tuple $\xbar\seq\zbar$ and a finite set $B\seq A$, let $S^{R}_{\xbar}(B)$ be the set of complete quantifier-free $\cL_R$-types realized in $\cA$, which are  in the variables $\xbar$, and over parameters from $B$.

Given $m\geq 1$, $\xbar\seq\zbar$ nonempty, and $B\seq A$ finite, we say that a type $p\in S^R_{\xbar}(B)$ \textbf{supports an $m$-array} if  $p$ has at least $m$ pairwise disjoint realizations in $\cA$ (where $\abar,\bbar\models p$ are \emph{disjoint} if $\abar\cap\bbar=\emptyset$). We say $R$ has \textbf{uniformly bounded arrays in $\cA$} if there are $m,N\in\N$ such that, for any nonempty tuple $\xbar\seq\zbar$ and any finite $B\seq A$, at most $N$ types in $S^R_{\xbar}(B)$ support an $m$-array.
\end{definition}

\begin{theorem}[\textnormal{Laskowski; Laskowski \& Terry}]\label{thm:MAfacts}
The following are equivalent.
\begin{enumerate}[$(i)$]
\item $\Th(\cA)$ is weakly minimal with trivial forking.\footnote{Recall that a (stable) theory has \emph{trivial forking} if, for any tuples $a,b,c$ and set $D$ in some model, if $\tp(a/Dbc)$ forks over $D$ then either $\tp(a/Db)$ or $\tp(a/Dc)$ forks over $D$ (see, e.g., \cite{JBGoode}).}
\item Every atomic $\cL$-formula is equivalent (modulo $\Th_{\cL_A}(\cA)$) to a Boolean combination of mutually algebraic $\cL_A$-formulas.
\item Every atomic $\cL$-formula has uniformly bounded arrays in $\cA$.
\end{enumerate}
\end{theorem}
\begin{proof}
See Proposition 2.7 and Theorem 3.3 of \cite{LaskMAS} for the equivalence of $(i)$ and $(ii)$. See \cite[Theorem 7.3]{LasTe1} for the equivalence with $(iii)$. 
\end{proof}

Following \cite{LaskMAS}, we call $\cA$ \textbf{mutually algebraic} if it satisfies the conditions in Theorem \ref{thm:MAfacts}. Each of these formulations of mutually algebraicity will be crucially used in our subsequent work. We will also need the following results. 

\begin{theorem}[\textnormal{Laskowski}]\label{thm:MAredexp}
If $\cA$ is mutually algebraic, and $\cM$ is a reduct of $\cA$ or an expansion of $\cA$ by unary predicates, then $\cM$ is mutually algebraic.
\end{theorem}
\begin{proof} 
This follows from Theorem 3.3 and Lemma 2.10 of \cite{LaskMAS}.  
\end{proof}

\begin{corollary}\label{cor:MAinter}
Suppose $\cA$ is mutually algebraic and $\cM$ is a first-order structure interpretable in $\cA$. If the universe of $\cM$ has $U$-rank $1$ as a definable set in $\cA^{\eq}$, then $\cM$ is mutually algebraic. 
\end{corollary}
\begin{proof}
Let $M$ be the universe of $\cM$, which we view as a definable set in $\cA^{\eq}$. Then $\cM$ is a reduct of the $\cA^{\eq}$-induced structure on $M$ and so, by Theorem \ref{thm:MAredexp}, we may assume $\cM=M_{\cA}$. Since $M$ is definable, it is bounded in $\cA^{\eq}$. Since $\cA^{\eq}$ is stable and $M$ has $U$-rank $1$ as a definable set, it follows that $\cM$ is weakly minimal (see, e.g., \cite[Theorem 2.10]{CoSS}). Since $\Th(\cA)$ has trivial forking, so does $\Th(\cA^{\eq})$ by \cite[Lemma 1]{JBGoode}. From this one can show that $\Th(\cM)$ has trivial forking (see, e.g., \cite[Proposition 2.7]{CoSS}). So $\cM$ is mutually algebraic by Theorem \ref{thm:MAfacts}. 
\end{proof}

\subsection{Small sets and nfcp}
Let $T$ be a complete $\cL$-theory. Recall that $T$ has \textbf{nfcp}  if for any $\cL$-formula $\phi(\xbar;\ybar)$ there is $k\geq 1$ such that, for any $\cM\models T$ and $B\seq M^{\ybar}$, the partial type $\{\phi(\xbar;\bbar):\bbar\in B\}$ is consistent if and only if it is $k$-consistent. 

\begin{fact}\label{fact:wmEinfty}$~$
\begin{enumerate}[$(a)$]
\item $T$ has nfcp if and only if it is stable and eliminates $\exists^\infty$ in all imaginary sorts.
\item If $T$ is weakly minimal then it has nfcp. 
\end{enumerate}
\end{fact}
\begin{proof}
Part $(a)$ is one of the equivalences of Shelah's \emph{fcp theorem} \cite[Theorem II.4.4]{Shbook}. For part $(b)$, it follows from Section 2 of \cite{GagGT} that any weakly minimal theory eliminates $\exists^\infty$ in all imaginary sorts. 
\end{proof}

Given $\cM\models T$ and $A\seq M$, Casanovas and Ziegler \cite{CaZi} provide a test for transferring nfcp from $\cM$ and $A_{\cM}$ to $(\cM,A)$. 

\begin{definition}\label{def:small}
Suppose $\cM\models T$ and $A\seq M$. Then $A\seq M$ is \textbf{small in $\cM$} if there is $(\cN,B)\equiv_{\cL(P)}(\cM,A)$ such that, for any finite tuple $\bbar$ from $N$, any type $p\in S_1^{\cL}(B\bbar)$ is realized in $\cN$. 
\end{definition}

The following is \cite[Proposition 5.7]{CaZi}.
  
 \begin{proposition}[Casanovas \& Ziegler]\label{prop:CZsmall}
 Suppose $\cM\models T$ and $A\seq M$ is small in $\cM$. Then $(\cM,A)$ has nfcp if and only if $\cM$ and $A_{\cM}$ have nfcp.
\end{proposition}

\begin{remark}\label{rem:small}
If $\cM\models T$ and $A\seq M$ is $\cM$-definable, then $A$ is small in $\cM$ if and only if it is finite.
\end{remark}

Recall that $T$ is \textbf{unidimensional} if any two non-algebraic stationary types are non-orthogonal. This notion is of interest to us due to the following standard fact (see, e.g., \cite[Remark 4.5.11]{Pibook}).

\begin{fact}\label{fact:WMGuni}
If $T$ is weakly minimal and expands a group then it is unidimensional.
\end{fact}

The next proposition gives a characterization of small sets in weakly minimal unidimensional theories, which generalizes a result of Casanovas and Ziegler for strongly minimal theories (see the remarks following \cite[Corollary 5.4]{CaZi}).

\begin{proposition}\label{prop:wmg-small}
Suppose $T$ is weakly minimal and unidimensional. Then $A\seq M$ is not small in $\cM$ if and only if there is a finite tuple $\mbar$ from $M$ such that if $(\cN,B)\succeq_{\cL(P)}(\cM,A)$, then $N=\acl_{\cL}(B\mbar)$. 
\end{proposition}
\begin{proof}
First assume $A\seq M$ is not small in $\cM$. Let $(\cN,B)\succeq_{\cL(P)}(\cM,A)$ be $|T_M|^+$-saturated, and let $(\cU,C)\succeq_{\cL(P)}(\cN,B)$ be $|N|^+$-saturated. By assumption on $A$, there is a finite tuple $\dbar\subset U$ and a type $p\in S_1^{\cL}(NC\dbar)$ such that $p$ is not realized in $\cU$. Let $p_0\in S_1^{\cL}(N)$ be the restriction of $p$. We claim that $p_0(U)\seq \acl_{\cL}(NC\dbar)$. Suppose otherwise that there is $a\in p_0(U)\backslash \acl_{\cL}(NC\dbar)$, and let $q=\tp_{\cL}(a/NC\dbar)$. Then $p$ and $q$ are non-algebraic types with the same restriction to $N$, and so $p=q$ by weak minimality and stationarity, which contradicts that $p$ is not realized in $\cU$. 

Next, we claim that $U=\acl_{\cL}(NC\dbar)$. So fix $b\in U\backslash N$, and let $q_0=\tp_{\cL}(b/N)$. Then $q_0$ is non-algebraic, and thus non-orthogonal to $p_0$ by unidimensionality. By saturation of $\cN$, $q_0$ is not weakly orthogonal to $p_0$. So there is $a\models p_0$ such that $b\in \acl_{\cL}(aN)$. By the above, $a\in p_0(U)\seq\acl_{\cL}(NC\dbar)$, and so $b\in\acl_{\cL}(NC\dbar)$. 

Altogether, $(\cU,C)$ omits the partial type
\[
\pi(x)=\left\{\forall\ybar\in P\left(\exists^{<\infty}u\,\phi(u;\ybar)\rightarrow\neg\phi(x;\ybar)\right):\phi(x;\ybar)\text{ is an $\cL_{N\dbar}$-formula}\right\}.
\]
By saturation of $(\cU,C)$, we may choose an $\cL$-formula $\phi(x;\ybar;\zbar)$, and some $\bbar\in (N\dbar)^{\zbar}$ such that $\phi(x;\cbar;\bbar)$ is algebraic for all $\cbar\in C^{\ybar}$, and $U=\bigcup_{\cbar\in C^{\ybar}}\phi(U;\cbar;\bbar)$. So the desired result follows from  $\cL(P)$-elementarity.  

Conversely, suppose $A\seq M$ satisfies the latter condition. As in the arguments above, we can use elimination of $\exists^\infty$ and saturated extensions of $(\cM,A)$ to express this assumption elementarily. In particular, for any $(\cN,B)\equiv_{\cL(P)}(\cM,A)$ there is a finite tuple $\cbar$ from $N$ such that $N=\acl_{\cL}(B\cbar)$, and so any non-algebraic $p\in S_1^{\cL}(B\cbar)$ is omitted in $\cN$. This shows that $A$ is not small in $\cM$. 
\end{proof}

\begin{remark}
In the previous result,  the assumption of unidimensionality cannot be removed. For example, let $T$ be the theory of an equivalence relation $E$ with two infinite classes. Fix $\cM\models T$ and let $A\seq M$ be one $E$-class. Then $T$ is weakly minimal (but not unidimensional), $A$ is not small in $\cM$ by Remark \ref{rem:small}, and if $(\cN,B)\equiv_{\cL_M(P)}(\cM,A)$ is $\aleph_1$-saturated then $N\neq \acl_{\cL}(BM)$. 
\end{remark}

Next, we refine Proposition \ref{prop:wmg-small} for the case of weakly minimal abelian groups with finite torsion. In particular, we show that small sets in such groups are characterized by a combinatorial sparsity condition analyzed in \cite{CoSS} for the case of $(\Z,+)$.

\begin{definition}
Let $(G,+)$ be an abelian group, and suppose $A\seq G$.
\begin{enumerate}
\item Given $n\in\Z$, let $nA=\{nx:x\in A\}$.
\item Given $n\geq 1$, let $\Sigma^{\pm}_n(A)=\{a_1+\ldots+a_n:a_1,\ldots,a_n\in A\cup\nv A\cup\{0\}\}$.
\item $A$ is \textbf{generic} if there is a finite set $F\seq G$ such that $G=A+F$. 
\end{enumerate}
\end{definition}

\begin{proposition}\label{prop:wmag-small}
Let $\cG$ be a weakly minimal abelian group with finite torsion. Then $A\seq G$ is small in $\cG$ if and only if $\Sigma^{\pm}_n(A)$ is non-generic for all $n\geq 1$.
\end{proposition}
\begin{proof}
Let $\cL$ denote the language of groups. Suppose $A\seq G$ is such that $G=\Sigma^{\pm}_n(A)+F$ for some $n\geq 1$ and finite $F\subset G$. For any $(\cN,B)\succeq_{\cL(P)}(\cG,A)$ we have $N=\Sigma^{\pm}_n(B)+F\seq\acl_{\cL}(B\cup F)$. So $A$ is not small in $\cG$ by Proposition \ref{prop:wmg-small}.

Conversely, suppose $A\seq G$ is not small in $\cG$. Let $(\cN,B)\succeq_{\cL(P)}(\cG,A)$ be saturated. By Proposition \ref{prop:wmg-small}, $N=\acl_{\cL}(B\cup F)$ for some finite $F\subset G$. Given $k,n\geq 1$, set $X_{k,n}=\{x\in N:mx\in\Sigma^{\pm}_n(B\cup F)\text{ for some }1\leq m\leq k\}$. 
Then each $X_{k,n}$ is $(\cN,B)$-definable, and $\acl_{\cL}(B\cup F)=\bigcup_{k,n}X_{k,n}$ by Proposition \ref{prop:wmagacl}. So $N=\bigcup_{n,k}X_{k,n}$ which, by saturation of $(\cN,B)$, implies that there are $k,n\geq 1$ such that $N=X_{k,n}$. By $\cL(P)$-elementarity, it follows that for any $x\in G$ there is $m\leq k$ such that $mx\in \Sigma^{\pm}_n(A\cup F)$. 

Now, setting $\ell=k!$, we have $\ell G\seq \sum^{\pm}_{\ell n}(A\cup F)\seq \sum^{\pm}_{\ell n}(A)+\sum^{\pm}_{\ell n}(F)$ by the above. So $\sum^{\pm}_{\ell n}(A)$ is generic since $\ell G$ is generic and $\sum^{\pm}_{\ell n}(F)$ is finite.\footnote{The authors are grateful to the referee for simplifying the end of this proof.}
\end{proof}

\begin{remark}\label{rem:SS}
Let $\cG$ be a weakly minimal abelian group with finite torsion, and suppose $A\seq G$. Proposition \ref{prop:wmag-small} shows that $A$ is small in $\cG$ if and only if the iterated sumsets $\Sigma^{\pm}_n(A)$ are ``algebraically small" in the sense of being non-generic. In fact, this can be strengthened to say that $A$ is small in $\cG$ if and only if $\Sigma^{\pm}_n(A)$ has (upper) Banach density $0$ for all $n\geq 1$. This follows from the standard exercise that if a set $X\seq G$ has positive Banach density, then $X-X$ is generic.  
\end{remark}

\begin{corollary}\label{cor:interpretG}
Let $\cG$ be a weakly minimal abelian group with finite torsion.
\begin{enumerate}[$(a)$]
\item If $A\seq G$ is not small in $\cG$ then $A_{\cG}$ interprets $\cG$.
\item Given $A\seq G$, if $\Th(A_{\cG})$ is stable with trivial forking then $A$ is small in $\cG$.
\end{enumerate}
\end{corollary}
\begin{proof}
Part $(a)$. Suppose $A\seq G$ is not small in $\cG$. By Proposition \ref{prop:wmag-small}, we may fix a finite set $F\subset G$ and some $n\geq 1$ such that $G= \Sigma^{\pm}_n(A)+F$. From this, a standard argument shows that $\cG$ is interpretable in $A_{\cG}$ (details are left to the reader). 

Part $(b)$ follows from part $(a)$ since a stable theory with trivial forking cannot interpret an infinite group. 
\end{proof}

\begin{remark}
We note that some of the previous results on small sets can be used to give more elementary proofs of special cases of the main theorems in Section \ref{sec:WMG}. 

Firstly, Proposition \ref{prop:wmg-small} yields an alternate proof that if $T$ is weakly minimal and unidimensional, $\cM\models T$, and $\cL=\cL_M$, then any $A\seq M$ is bounded in $\cM$  (a special case of Theorem \ref{thm:wmt-bounded}). The argument splits into two cases. If $A$ is small in $\cM$ then it is bounded in $\cM$ by \cite[Proposition 2.1]{CaZi} and Fact \ref{fact:wmEinfty}$(b)$. If $A$ is not small in $\cM$ then it is bounded in $\cM$ by Propositions \ref{prop:5.3} and \ref{prop:wmg-small}.

Secondly, if $A\seq\Z$ then $A$ satisfies the conclusion of Proposition \ref{prop:wmag-small} if and only if $m\Z\not\seq \Sigma^{\pm}_n(A)$ for all $m,n\geq 1$ (see \cite[Section 4]{CoSS}, where such sets are called \emph{sufficiently sparse}). It follows that if $A\seq\Z$ then either $A$ is small in $(\Z,+)$ or $N=\acl_{\cL}(B)$ for any $(\cN,B)\equiv_{\cL(P)}(\Z,+,A)$ (where $\cL$ is just the group language). So an argument similar to the one above yields an alternate proof that all subsets of $\Z$ are bounded in $(\Z,+)$, which is a special case of Theorem \ref{thm:noconstants}.
\end{remark}

\subsection{Conclusion}

We now state and prove the main result of Section \ref{sec:WMAG}.

\begin{theorem}\label{thm:MAG-main}
Let $\cG=(G,+)$ be a weakly minimal abelian group. 
Fix $A\seq G$, and suppose $A_{\cG}^{\qf}$ is mutually algebraic.  Then, for any finite $F\subset G$ and any $B\seq A+F$,  $(\cG,B)$ is superstable of $U$-rank at most $\omega$. Moreover, if $\cG$ has finite torsion then $(\cG,B)$ has nfcp; and if $\cG=(\Z,+)$ and $B$ is infinite then $(\cG,B)$ has $U$-rank $\omega$. 
\end{theorem}
\begin{proof}
We may assume $A$ is infinite. By Proposition \ref{prop:wmagind} and Theorem \ref{thm:MAredexp}, $A_{\cG}$ is mutually algebraic. Fix a finite set $F\subset G$. Then $(A+ F)_{\cG}$ is interpretable in $A_{\cG}$ as a structure on $(A\times F)/E$, where $E$ is the $A_{\cG}$-definable equivalence relation on $A\times F$ given by $a_1+f_1=a_2+f_2$.   Since $F$ is finite, $(A\times F)/E$ has $U$-rank $1$ as an interpretable set in $A_{\cG}$. So $(A+ F)_{\cG}$ is mutually algebraic by Corollary \ref{cor:MAinter}. 

Now fix $B\seq A+F$. Then $\cA:=((A+F)_{\cG},B)$ is mutually algebraic by Theorem \ref{thm:MAredexp}. Therefore $B$ has $U$-rank $1$ as an $\cA$-definable set. Since $B_{\cG}$ is interpretable in $\cA$ as a structure with universe $B$, we conclude from Corollary \ref{cor:MAinter} that $\cB_{\cG}$ is mutually algebraic. By Theorem \ref{thm:wmt-main}, $(\cG,B)$ is superstable of $U$-rank at most $\omega$. 

Finally, if $G$ has finite torsion then $B$ is small in $\cG$ by Corollary \ref{cor:interpretG}$(b)$ and Theorem \ref{thm:MAfacts}. So $(\cG,B)$ has nfcp by Fact \ref{fact:wmEinfty}$(b)$ and Proposition \ref{prop:CZsmall}. Note also that if $B$ is infinite then it is not $\cG$-definable by Remark \ref{rem:small}. So if $\cG=(\Z,+)$ and $B$ is infinite then $(\cG,B)$ does not have finite $U$-rank by \cite[Theorem 1]{PaSk}.
\end{proof}

\begin{remark}\label{rem:subgroup}
Suppose $\cK=(K,+)$ is an abelian group and $A\seq K$ is such that $A^{\qf}_{\cK}$ is mutually algebraic. Let $\cG=(G,+)$ be a subgroup of $\cK$, such that $A\seq G$. Then $A^{\qf}_{\cG}$ is a reduct of $A^{\qf}_{\cK}$ and so, if $\cG$ is weakly minimal, then the conclusion of Theorem \ref{thm:MAG-main} holds.
\end{remark}

In all examples of stable structures of the form $(\Z,+,A)$, considered in \cite{CoSS}, \cite{CoGG}, \cite{LaPo}, and \cite{PaSk}, the induced structure $A_{(\Z,+)}$ is mutually algebraic, and so these results fall under the umbrella of Theorem \ref{thm:MAG-main}. However, in many of these results, a considerable amount of work is still required to demonstrate mutual algebraicity for $A_{(\Z,+)}$. Indeed, this is usually done by showing that $A_{(\Z,+)}$ is interpretable in a more familiar structure (e.g., $\N$ with the successor function). As we will see later, there are some cases where it is significantly easier to just show $A_{(\Z,+)}$ is mutually algebraic. Theorem \ref{thm:FGM} is a notable example. Moreover, as shown in the proof of Theorem \ref{thm:MAG-main}, once one knows that $A_{(\Z,+)}$ is mutually algebraic, it quickly follows that $B_{(\Z,+)}$ is mutually algebraic for any $B\seq A+F$ where $F\subset\Z$ is finite. This also eliminates a nontrivial amount of technical and tedious work in some  examples considered in the sources above (e.g., \cite[Lemma 4.17]{CoGG}).

\section{Stable expansions of weakly minimal abelian groups}\label{sec:Zexp}

In this section, we give several new families of stable expansions of weakly minimal abelian groups. The main results are Theorems \ref{thm:Kepler}, \ref{thm:FGM}, and \ref{thm:LRR}. Each one of these theorems is formulated for a weakly minimal abelian group $\cG=(G,+)$ satisfying certain further properties, which always hold for $(\Z,+)$. The conclusion of each of these theorems is that some expansion of the form $(\cG,B)$ is superstable of $U$-rank  \emph{at most} $\omega$. For each result, we obtain this by showing that the induced structure $B_{\cG}$ is mutually algebraic. Therefore, if $\cG=(\Z,+)$ and $B$ is infinite, then $(\cG,B)$ has $U$-rank \emph{exactly} $\omega$ by Theorem \ref{thm:MAG-main}. 

Given an integer $n\geq 1$, we let $[n]=\{1,\ldots,n\}$.

\subsection{Strongly lacunary sets in $\C$}\label{sec:Kepler}

An increasing sequence $(a_n)_{n=0}^\infty$ in $\R$ is often called \emph{lacunary} if $\liminf_{n\to\infty}\frac{a_{n+1}}{a_n}>1$. This motivates the following definition.

\begin{definition}
Suppose $A\seq\C$ is countable. Then $A$ is \textbf{strongly lacunary} if there is an enumeration $\{a_n\}_{n=0}^\infty$ of $A\backslash\{0\}$ such that either $\lim_{n\to\infty}|\frac{a_{n+1}}{a_n}|=\infty$ or $\lim_{n\to\infty}\frac{a_{n+1}}{a_n}=\kappa$ for some $\kappa\in\C$ such that $|\kappa|>1$. 
\end{definition}

Suppose $A\seq \C$ is strongly lacunary, witnessed by an enumeration $\{a_n\}_{n=0}^\infty$. Then there is some $N\geq 0$ such that $|a_{n+1}|>|a_n|$ for all $n\geq N$. It follows from this that if $\{c_n\}_{n=0}^\infty$ is another enumeration witnessing that $A$ is strongly lacunary, then either $\lim_{n\to\infty}|\frac{a_{n+1}}{a_n}|=\infty=\lim_{n\to\infty}|\frac{c_{n+1}}{c_n}|$, or $\lim_{n\to\infty}\frac{a_{n+1}}{a_n}=\kappa=\lim_{n\to\infty}\frac{c_{n+1}}{c_n}$ with $|\kappa|>1$. In the former case we call $A$ \textbf{divergent}, and in the latter case we call $A$ \textbf{convergent} and call $\kappa$ the \textbf{Kepler limit of $A$} (this terminology is often used in the context of Fibonacci sequence, whose Kepler limit is the golden ratio).

In \cite[Theorem 7.16$(a)$]{CoSS} the first author showed that \emph{any} divergent strongly lacunary set $A\seq\Z^+$ yields a stable expansion $(\Z,+,A)$ (this was shown independently by Lambotte and Point \cite{LaPo} under the extra assumption that the set is eventually periodic modulo any $n\geq 1$). We will reprove this below in a more general setting. On the other hand, there are strongly lacunary sets $A\seq\Z$ such that $(\Z,+,A)$ is unstable (the existence of such sets was questioned in  \cite{CoSS} and \cite{LaPo}). For example, given  $q\geq 2$, if $A_q=\{q^n+n:n\in\N\}$, then $(\Z,+,A_q)$ is interdefinable with $(\Z,+,<,x\mapsto q^x)$ (see \cite[Theorem 4.8]{CoGG}). The proof generalizes to $\{F_n+n:n\in\N\}$, where $F_n$ is the $n^{\text{th}}$ Fibonacci number, and so we also have a  strongly lacunary set $A\seq\Z^+$, with an irrational Kepler limit, such that $(\Z,+,A)$ is unstable. In this section, we show that this cannot happen for a strongly lacunary set with a transcendental Kepler limit. This generalizes a similar result of Lambotte and Point \cite{LaPo} for certain expansions of $(\Z,+)$.

\begin{lemma}\label{lem:trans-induced}
Suppose $A\seq \C$ is strongly lacunary, and either divergent or convergent with transcendental Kepler limit. Then $A_{(\C,+)}^{\qf}$ is interdefinable with $A$ in the language of equality.
\end{lemma}
\begin{proof}
The proof uses techniques similar to those of Palac\'{i}n and Sklinos \cite{PaSk} and Lambotte and Point \cite{LaPo}.  Let $A=\{a_n\}_{n=0}^\infty$ be an enumeration of $A$ witnessing that $A$ is strongly lacunary. We may assume $|a_{n+1}|>|a_n|$ for all $n\in\N$. 

Let $\cN$ be the structure on $\N$ induced from $A_{(\C,+)}^{\qf}$ via the map $a_n\mapsto n$. It suffices to show $\cN$ is interdefinable with the structure $\N$ in the language of equality, which we denote by $\dot{\N}$. Given $k\geq 1$, $\dbar\in \Z^k$, and $r\in \C$, define
\[
X_{\dbar;r}=\left\{\nbar\in\N^k:n_i\neq n_j\text{ for all distinct $i,j\in[k]$ and $\textstyle\sum_{i=1}^k d_ia_{n_i}=r$}\right\}.
\]
Note that any $X_{\dbar;r}$ is $\emptyset$-definable in $\cN$. Let $\cN_0$ be the reduct of $\cN$ to symbols for  $X_{\dbar;r}$, where $\dbar\in (\Z^*)^k$ and $r\in \C$. It is easy to see that $\cN$ is interdefinable with $\cN_0$, and so it suffices to show that $\cN_0$ is interdefinable with $\dot{\N}$. Fix $k\geq 1$, $\dbar\in (\Z^*)^k$, and $r\in \C$. Toward a contradiction, suppose $X_{\dbar;r}$ is infinite. 

By pigeonhole, there are infinitely many tuples in $X_{\dbar;r}$ of the same order type. After permuting the coordinates, we may fix an infinite sequence $(\nbar(t))_{t=0}^\infty$ from $X_{\dbar;r}$, such that $n(t)_1<\ldots<n(t)_k$ for all $t\in\N$. Since $(\nbar(t))_{t=0}^\infty$ is infinite, we may pass to a subsequence and assume that $(n(t)_k)_{t=0}^\infty$ diverges. For $t\in\N$ and $i\in[k]$, let $u(t)_i=n(t)_k-n(t)_i$. Then $u(t)_1>\ldots>u(t)_k$ for all $t\in\N$. Let $u_k=0$, and note that $u(t)_k=u_k$ for all $t\in\N$. If the sequence $(u(t)_{k-1})_{t=0}^\infty$ does not diverge then, by pigeonhole, it contains a constant subsequence. So, after passing to a subsequence, we may assume that either $(u(t)_{k-1})_{t=0}^\infty$ diverges, or $u(t)_{k-1}=u_{k-1}$ for all $t\in\N$ and some $u_{k-1}\in\N$. Repeating this process, we may assume that for some $\ell\in[k]$ and $u_k,u_{k-1},\ldots,u_\ell\in\N$, we have $u(t)_i=u_i$ for all $t\in\N$ and $\ell\leq i\leq k$, and $\lim_{t\to\infty}u(t)_i=\infty$ for all $1\leq i<\ell$ (note that $\ell=1$ is possible, making the second condition vacuous). 

For any $1\leq i<\ell$, since $(u(t)_i)_{t=0}^\infty$ diverges, we have that, for any $u\in\N$,
\[
0\leq \lim_{t\to\infty}\frac{|a_{n(t)_i}|}{|a_{n(t)_k}|}=\lim_{t\to\infty}\frac{|a_{n(t)_k-u(t)_i}|}{|a_{n(t)_k}|}\leq \lim_{t\to\infty}\frac{|a_{n(t)_k-u}|}{|a_{n(t)_k}|}.
\]
So we have $\lim_{t\to\infty}\frac{a_{n(t)_i}}{a_{n(t)_k}}=0$ for all $1\leq i<\ell$ (if $A$ is divergent this is clear, and if $A$ is convergent then this follows from $|\tau|>1$). Recall that $\nbar(t)\in X_{\dbar;r}$ for all $t\in\N$, and that $(n(t)_k)_{k=0}^\infty$ diverges. Altogether,
\begin{equation*}
0=\lim_{t\to\infty}\frac{r}{a_{n(t)_k}}=\lim_{t\to\infty}\sum_{i=1}^k d_i\frac{a_{n(t)_i}}{a_{n(t)_k}}=\lim_{t\to\infty}\sum_{i=\ell}^k d_i\frac{a_{n(t)_k-u_i}}{a_{n(t)_k}}.\tag{$\dagger$}
\end{equation*}
Recall that $u_\ell>\ldots>u_k=0$. Therefore,  if $A$ is divergent then the rightmost limit in $(\dagger)$ is $d_k$, and if $A$ is convergent then the rightmost limit in $(\dagger)$ is $\sum_{i=\ell}^k d_i\tau^{\nv u_i}$. In either case, this contradicts $d_i\neq 0$ for all $\ell\leq i\leq k$. 
\end{proof}

\begin{theorem}\label{thm:Kepler}
Suppose $\cG=(G,+)$ is a weakly minimal subgroup of $(\C,+)$, and $A\seq G$ is strongly lacunary and either divergent or convergent with transcendental Kepler limit. Then, for any finite $F\seq G$ and infinite $B\seq A+F$, $(\cG,B)$ has nfcp and is superstable of $U$-rank at most $\omega$.
\end{theorem}
\begin{proof}
Apply Lemma \ref{lem:trans-induced} and Theorem \ref{thm:MAG-main} (via Remark \ref{rem:subgroup}).
\end{proof}

\subsection{Finite rank multiplicative groups}\label{sec:FGM}

Throughout this section, we fix an algebraically closed field $\K$ of characteristic $0$ and a subgroup $\cG=(G,+)$ of the additive group $(\K,+)$. Let $\K^*$ denote the multiplicative subgroup of nonzero elements of $\K$.
Recall that the \emph{rank} of an abelian group is the cardinality of a maximal $\Z$-linearly independent set. We will give a short proof of the following theorem.

\begin{theorem}\label{thm:FGM}
Suppose $\cG$ is weakly minimal, and $A\seq G$ is contained in a finite rank subgroup of $\K^*$. Then, for any finite $F\subset G$ and any $B\seq A+F$, $(\cG,B)$ has nfcp and is superstable of $U$-rank at most $\omega$. 
\end{theorem}

For the case $\cG=(\Z,+)$, this was proved by the first author in \cite[Theorem 3.1]{CoGG} (although explicitly only for $A\seq\Z^+$ and $F=\{0\}$). The proof relies on results concerning the structure of solutions to linear equations from finite rank multiplicative groups. This goes back to work of Mann, and is connected to number-theoretic results around Lang's Conjecture (proved by Faltings and Vojta) and Schmidt's Subspace Theorem. See \cite{PiLC} for a model-theoretic account of this relationship.

 In \cite{BeZi}, Belegradek and Zilber use these type of results to prove stability for the expansion of the \emph{field} $(\C,+,\cdot)$ by a finite rank multiplicative subgroup of the unit circle. Similar results for arbitrary finite rank subgroups of $\C^*$ were proved by Van den Dries and G\"{u}nayd{\i}n \cite{vdDGu}. Note however that the full conclusion of Theorem \ref{thm:FGM} does not hold for expansions of fields. For instance, if $\Gamma=\{2^n:n\in\Z\}$ and $\Pi=\{2^n:n\in\N\}$, then  $(\C,+,\cdot,\Gamma)$ is stable while $(\C,+,\cdot,\Pi)$ defines the ordering on $\Pi$. Note also that $\Gamma_{(\C,+,\cdot)}$ is interdefinable with $(\Z,+)$, and thus is weakly minimal but does not have trivial forking.

The work in \cite{CoGG} uses the following result, which is \cite[Theorem 1.1]{ESS}.

\begin{theorem}[Evertse, Schlickewei, Schmidt]\label{thm:ESS}
Suppose $\Gamma$ is a  subgroup of $(\K^*)^k$ of  rank at most $\rho$, for some $k,\rho\in\N$. Then there is an integer $N=N(k,\rho)$ such that, for any $c_1,\ldots,c_k\in \K$ and any $r\in \K^*$, there are at most $N$ tuples $(x_1,\ldots,x_k)\in\Gamma$ such that $c_1x_1+\ldots+c_kx_k=r$ and $\sum_{i\in I}c_ix_i\neq 0$ for all nonempty $I\seq [k]$.
\end{theorem}

We will use this result to directly show that, for $A\seq G$ as in Theorem \ref{thm:FGM}, $A_{\cG}^{\qf}$ is mutually algebraic.

\begin{proof}[Proof of Theorem \ref{thm:FGM}]
Let $A\seq G$ be as in the theorem. We assume $A$ is infinite. By Theorem \ref{thm:MAG-main}, it suffices to show $A_{\cG}^{\qf}$ is mutually algebraic. Given $k\geq 1$, $\cbar\in\{1,\nv 1\}^k$, and $r\in G$, define $A(\cbar;r):=\{\abar\in A^k:c_1a_1+\ldots+c_ka_k=r\}$ and define $A_0(\cbar;r)$ to be the set of $\abar\in A(\cbar;r)$ such that $\sum_{i\in I}c_ia_i\neq 0$ for all nonempty $I\subsetneq [k]$. Note that any $A(\cbar;r)$ is a finite Boolean combination of sets of the form $A_0(\cbar';r')$ for some $k'$-tuple $\cbar'$ and $r'\in G$. So it suffices to show that, for any $k\geq 1$, $\cbar\in\{1,\nv 1\}^k$, and $r\in G$, $A_0(\cbar;r)$ is  mutually algebraic as a $k$-ary relation on $A$. 

Fix $k\geq 1$, $\cbar\in\{1,\nv 1\}^k$, and $r\in G$. Suppose $A\seq\Gamma$, where $\Gamma$ is a subgroup of $\K^*$ of rank $\rho\in\N$.
Note that $\Gamma^k$ is a subgroup of $(\K^*)^k$ of rank $k\rho$.  Let $\Gamma_0(\cbar;r)$ be the set of $\xbar\in\Gamma^k$ such that $c_1x_1+\ldots+c_kx_k=r$ and $\sum_{i\in I}c_ix_i\neq 0$ for all nonempty $I\subsetneq [k]$.  We have $A_0(\cbar;r)\seq \Gamma_0(\cbar;r)$, and so if $r\in\K^*$ then $A_0(\cbar;r)$ is finite by Theorem \ref{thm:ESS}. So we may assume $r=0$. Given $i\in [k]$, set $\Gamma_{0,i}=\Gamma_0(c_1,\ldots,c_{i-1},c_{i+1},\ldots,c_k;\nv c_i)$. By Theorem \ref{thm:ESS}, there is some $N\geq 0$ such that $|\Gamma_{0,i}|\leq N$ for all $i\in[n]$. Fix $i\in[n]$ and $b\in A$ and set
\[
X=\left\{(a_1,\ldots,a_{k-1})\in A^{k-1}:(a_1,\ldots,a_{i-1},b,a_i,\ldots,a_{k-1})\in A_0(\cbar;0)\right\}.
\]  
Then $b\inv X\seq \Gamma_{0,i}$, and so $|X|\leq N$, as desired. 
\end{proof}

\subsection{Linear recurrence relations}\label{sec:LRR}

In this section, we consider sets of algebraic numbers, which are enumerated by linear homogeneous recurrence relations with constant coefficients. Once again, our main task will be to show mutual algebraicity of $A_{\cG}$, where $A$ and $\cG$ satisfy certain assumptions. Toward this end, we first generalize the behavior found in Theorem \ref{thm:ESS} in order to formulate a combinatorial criterion for generating mutually algebraic structures. In contrast to Section \ref{sec:FGM} however, we will need to use  the characterization of mutual algebraicity involving uniformly bounded arrays.

Fix an infinite set $A$, an abelian group $\cG=(G,+)$, and a set $\Phi$ of functions from $A$ to $G$. (We do not assume $A\seq G$.)

\begin{definition}\label{def:ESS}$~$
\begin{enumerate}
\item Given $k\geq 1$, $\bar{\varphi}\in \Phi^k$, and $r\in G$, define
\begin{align*}
A(\bar{\varphi};r) &:= \left\{\abar\in A^k:\varphi_1(a_1)+\ldots+\varphi_k(a_k)=r\right\}, \text{ and }\\
A_0(\bar{\varphi};r) &: =\left\{\abar\in A(\bar{\varphi};r):\sum_{i\in I}\varphi_i(a_i)\neq 0\text{ for all }\emptyset\neq I\subsetneq [k]\right\}.
\end{align*}
\item Let $A_{\cG}^{\Phi}$ be  the relational structure with universe $A$  and a $k$-ary relation $R_{\bar{\varphi};r}$ interpreted as $A(\bar{\varphi};r)$, for all $k\geq 1$, $\bar{\varphi}\in\Phi^k$, and $r\in G$. 
\end{enumerate} 
\end{definition}

The following is the main technical lemma.

\begin{lemma}\label{lem:ESSMA}
Suppose that, for any $k\geq 1$, there is $n_k\in\N$  such that $|A_0(\bar{\varphi};r)|\leq n_k$ for any $\bar{\varphi}\in \Phi^k$ and $r\neq 0$. Then $A_{\cG}^{\Phi}$ is mutually algebraic.
\end{lemma}
\begin{proof}
Given $k\geq 1$, $\bar{\varphi}\in\Phi^k$, $r\in G$, $\xbar\seq\zbar=(z_1,\ldots,z_k)$, and finite $B\seq A$, set
\[
S^{\bar{\varphi};r}_{\xbar}(B)= S^{R_{\bar{\varphi};r}}_{\xbar}(B)
\]
(working in $A^\Phi_{\cG}$). We show, by induction on $k\geq 1$, that there are $m_k,N_k\in\N$ such that, for any $\bar{\varphi}\in\Phi^k$, $r\in G$, finite $B\seq A$, and any nonempty  $\xbar\seq \zbar=(z_1,\ldots,z_k)$, there are at most $N_k$ types in $S^{\bar{\varphi};r}_{\xbar}(B)$ supporting an $m_k$-array. By Theorem \ref{thm:MAfacts}, this suffices to show that $\cA^{\Phi}_{\cG}$ is mutually algebraic.

For the base case $k=1$, note that any unary relation $R(z)$ has uniformly bounded arrays. Indeed, given finite $B\seq A$, there are at most two types in $S^R_z(B)$ which contain $z\neq b$ for all $b\in B$. So fix $k>1$ and suppose we have defined $m_{k-1}$ and $N_{k-1}$ satisfying the desired properties. Let $\zbar=(z_1,\ldots,z_k)$. Given $\bar{\varphi}\in\Phi^k$, $r\in G$,  finite $B\seq A$, $\xbar\seq\zbar$, and an equivalence relation $E$ on $\xbar$, let $S_{\xbar,E}^{\bar{\varphi};r}(B)$ be the set of  $p\in S_{\xbar}^{\bar{\varphi};r}(B)$ such that:
\begin{enumerate}[$(i)$]
\item $z_i\neq b\in p$ for all $z_i\in\xbar$ and $b\in B$, and
\item given $z_i,z_j\in\xbar$, $z_i=z_j\in p$ if and only if $E(z_i,z_j)$.
\end{enumerate}
We claim that it suffices to find $m_k^*$ and $N_k^*$ such that, for any $\bar{\varphi}\in\Phi^k$, $r\in G$,  finite $B\seq A$, nonempty $\xbar\seq\zbar$, and any equivalence relation $E$ on $\xbar$, at most $N_k^*$ types in $S^{\bar{\varphi};r}_{\xbar,E}(B)$ support an $m_k^*$-array. Indeed, there are only finitely many choices for $\xbar$ and $E$; and if $p\in S_{\xbar}^{\bar{\varphi};r}(B)$ is such that $z_i=b\in p$ for some $z_i\in\xbar$ and $b\in B$, then $p$ cannot support a $2$-array. Therefore, setting $m_k=\max\{m_k^*,2\}$ and $N_k=hN_k^*$, where $h$ is the number of pairs $(\xbar, E)$ as above, it follows that $m_k$ and $N_k$ satisfy the desired properties. Define
\begin{align*}
N^*_k &= 2+\max\{N_{k-1}(2^\ell-2):1\leq \ell\leq k\}\text{ and}\\
m^*_k &= 1+\max\{n_{\ell}+(m_{k-1}-1)(2^\ell-2):1\leq \ell\leq k\}.
\end{align*}
Fix $\bar{\varphi}\in\Phi^k$, $r\in G$,  finite $B\seq A$, $\xbar\seq\zbar$ nonempty, and an equivalence relation $E$ on $\xbar$. Let $S^*$ be the set of  types in $S^{\bar{\varphi};r}_{\xbar,E}(B)$ that support an $m^*_k$-array. We want to show $|S^*|\leq N^*_k$.

For $\ubar\seq\zbar$ and $\abar\in A^{\ubar}$, let $\Sigma_{\ubar}\abar$ denote $\sum_{z_i\in \ubar}\varphi_i(a_i)$. Let $\ybar=\zbar\backslash \xbar$. Given $t\in G$, let $S^*(t)$ be the set of types  $p\in S^*$ such that $p\models R_{\bar{\varphi};r}(\xbar;\bbar)$ for some $\bbar\in B^{\ybar}$ satisfying $\Sigma_{\ybar}\bbar = t$. We claim that $|S^{*}(t)|\leq 1$ for any $t\in G$. Indeed, suppose we have $p,q\in S^*(t)$ for some $t\in G$. By construction, $p$ and $q$ agree on atomic formulas in the language of equality. So we just need to show that they agree on instances of $R_{\bar{\varphi};r}(\xbar;\ybar)$. Let $\abar^1,\abar^2\in A^{\xbar}$ realize $p$ and $q$, respectively. Since $p,q\in S^*(t)$, we have $r-\Sigma_{\xbar}\abar^1=t=r-\Sigma_{\xbar}\abar^2$. Given $\dbar\in B^{\ybar}$, we have 
\[
\textstyle p\models R_{\bar{\varphi};r}(\xbar;\dbar) \Leftrightarrow \Sigma_{\ybar}\dbar=r-\Sigma_{\xbar}\abar^1 \Leftrightarrow \Sigma_{\ybar}\dbar=r-\Sigma_{\xbar}\abar^2\Leftrightarrow q\models R_{\bar{\varphi};r}(\xbar;\dbar). 
\]
Altogether, we have $p=q$.

Let $X=\{t\in G:S^*(t)\neq\emptyset\}$ and, for $t\in X$, let $q_t$ be the unique type in $S^*(t)$. Note that there is at most one type in $S^{\bar{\varphi};r}_{\xbar,E}(B)$ which contains $\neg R_{\bar{\varphi};r}(\xbar;\bbar)$ for all $\bbar\in B^{\ybar}$. Altogether, $|S^*|\leq |X|+1$. So, to finish the proof, it suffices to show $|X\backslash\{r\}|\leq N^*_k-2$. Suppose, for a contradiction, that we have pairwise distinct $t_1,\ldots,t_M\in X\backslash\{r\}$, where $M:=N^*_k-1$. For $1\leq i\leq M$, let $p_i=q_{t_i}$. 

Let $\ell=|\xbar|$. 
Fix $i\in [M]$. Since $p_i\in S^*$, we may fix pairwise disjoint realizations $\abar^1,\ldots,\abar^{m^*_k}$ of $p_i$ in $A^{\xbar}$. Moreover, there is $\bbar^i\in B^{\ybar}$ such that $\Sigma_{\ybar}\bbar^i=t_i$ and $p_i\models R_{\bar{\varphi};r}(\xbar;\bbar^i)$. So we have $\Sigma_{\xbar}\abar^j=s_i:=r-t_i$ for all $j\in [m^*_k]$.  In particular, $\abar^1,\ldots,\abar^{m^*_k}\in A((\varphi_j)_{z_j\in\xbar};s_i)$. Since $s_i\neq 0$, we have $|A_0((\varphi_j)_{z_j\in\xbar};s_i)|\leq n_\ell$ and so, after renaming the tuples, we may assume $\abar^1,\ldots,\abar^m\not\in A_0((\varphi_j)_{z_j\in\xbar};s_i)$, where $m:=m^*_k-n_\ell\geq 1+(m_{k-1}-1)(2^\ell-2)$. Let $\Omega$ be the set of nonempty proper subtuples of $\xbar$, and note that $|\Omega|=2^\ell-2$. For each $j\in [m]$, there are $\xbar^{i,j}\in\Omega$ such that $\Sigma_{\xbar^{i,j}}(a^j_l)_{z_l\in\xbar^{i,j}}=0$. Since $m\geq 1+(m_{k-1}-1)(2^\ell-2)$, there are $\xbar'\in\Omega$ and $I\seq[m]$ such that $|I|= m_{k-1}$ and $\xbar^{i,j}=\xbar'$ for all $j\in I$. After renaming tuples, we may assume $I=[m_{k-1}]$. Set $\xbar^i=\xbar\backslash\xbar'\in\Omega$. For $j\in [m_{k-1}]$, let $\abar^{j}_*=(a^j_l)_{z_l\in \xbar^i}$. Then $\abar^1_*,\ldots,\abar^{m_{k-1}}_*$ are pairwise disjoint tuples, which all realize the same type $p^*_i\in S^{\bar{\varphi}^i;r}_{\xbar^i}(B)$, where $\bar{\varphi}^i=(\varphi_j)_{z_j\in \xbar^i}$. So $p^*_i$ supports an $m_{k-1}$-array. Note also that $p^*_i\models R_{\bar{\varphi}^i;r}(\xbar^i,\bbar^i)$. 

Since $M\geq 1+N_{k-1}(2^\ell-2)$, there are $\xbar^*\in\Omega$ and $I\seq[M]$ such that $|I|=N:=N_{k-1}+1$ and $\xbar^i=\xbar^*$ for all $i\in I$. After renaming the types, we may assume $I=[N]$. Let $\bar{\varphi}^*=(\varphi_j)_{z_j\in\xbar^*}$. Then $p^*_1,\ldots,p^*_{N}$ are types in $S^{\bar{\varphi}^*;r}_{\xbar^*}(B)$. For each $i\in[N]$, we have $p^*_i\models R_{\bar{\varphi}^*;r}(\xbar^*,\bbar^i)$ and $\Sigma_{\ybar}\bbar^i=t_i$. So, if $i,j\in[N]$ are distinct, then $p^*_i\neq p^*_j$ since $t_i\neq t_j$. So we have $N$ types in $S^{\bar{\varphi}^*;r}_{\xbar^*}(B)$, each of which supports an $m_{k-1}$-array. This is a contradiction, since $N=N_{k-1}+1$ and $|\xbar^*|\leq k-1$.
\end{proof}

\begin{example}
Note that if $\cG=(G,+)$ is an abelian group, $A\seq G$, and $\Phi$ consists of the maps $x\mapsto x$ and $x\mapsto\nv x$, then $A^{\Phi}_{\cG}$ is interdefinable with $A^{\qf}_{\cG}$. With this choice of $\Phi$, one can use Lemma \ref{lem:ESSMA} to prove mutual algebraicity of $A^{\qf}_G$, where $\cG$ and $A\seq G$ are as in Theorem \ref{thm:Kepler} or Theorem \ref{thm:FGM}. (In the latter case this claim is immediate from Theorem \ref{thm:ESS}; we leave the former case as an exercise.)
\end{example}

We now consider linear recurrence relations.
Let $\Q^{\alg}$ denote the field of algebraic numbers. 
We say that a set $A\seq \Q^{\alg}$ is \textbf{enumerated by a linear recurrence relation} if $A$ is enumerated by a sequence $(a_n)_{n=0}^\infty$ such that, for some $d\geq 1$ and $\beta_1,\ldots,\beta_d\in \Q^{\alg}$, we have 
\[
a_{n+d}=\beta_1a_{n+d-1}+\ldots+\beta_da_n
\]
for any $n\in\N$.  In this case, the \textbf{characteristic polynomial} of $A$, denoted $p_A(x)$, is $x^d-\beta_1x^{d-1}-\ldots-\beta_{d-1}x-\beta_d$. We assume $d$ is minimal, so $p_A(x)$ is uniquely determined. In particular, $\beta_d\neq 0$, and so $0$ is not a root of $p_A(x)$.  Let $\mu_1,\ldots,\mu_{d_*}\in \Q^{\alg}$ be the  distinct roots of $p_A(x)$ where $d_*\leq d$. By the general theory, there are nonzero polynomials $\alpha_1(x),\ldots,\alpha_{d_*}(x)\in\Q^{\alg}[x]$ such that $\alpha_i(x)$ has degree strictly less than the multiplicity of $\mu_i$ as a root of $p_A(x)$, and, for any $n\in\N$,
\[
a_n=\alpha_1(n)\mu_1^n+\ldots+\alpha_{d_*}(n)\mu_{d_*}^n.
\]
As the set $A$ is completely determined by $\beta_1,\ldots,\beta_d$, and $a_0,\ldots,a_{d-1}$, we sometimes identify $A$ with the notation $\text{LRR}(\beta_1,\ldots,\beta_d;a_0,\ldots,a_{d-1})$.

We are interested in stable expansions of weakly minimal subgroups of $\Q^{\alg}$ by sets enumerated by a linear recurrence relation. For $(\Z,+)$, the previous literature on this question is as follows. Palac\'{i}n \& Sklinos \cite{PaSk} and Poizat \cite{PoZ} proved stability for the expansion of $(\Z,+)$ by $\{q^n:n\in\N\}=\text{LRR}(q;1)$, where $q\geq 2$.\footnote{The first author recently learned that this also follows from an older result of Moosa and Scanlon, namely, \cite[Theorem 6.11]{MooScan}.} In \cite{CoSS}, the first author proved stability of $(\Z,+,A)$, for any $A\seq\Z$ enumerated by linear recurrence relation such that $p_A(x)$ is separable (so $d_*=d$), and there is some $1\leq t\leq d$ such that $\mu_t\in\R_{>1}$ and $|\mu_i|\leq 1$ for all $i\neq t$ (e.g., the Fibonacci sequence $\text{LRR}(1,1,0,1)$). In \cite{LaPo}, Lambotte and Point proved stability for the case when $p_A(x)$ is irreducible over $\Q$ and there is some $1\leq t\leq d$ such that $\mu_t\in\R_{>1}$ and $|\mu_i|<|\mu_t|$ for all $i\neq t$. There are also well-known examples of unstable expansions of $(\Z,+)$ by linear recurrences. For instance, given $k\geq 1$, the set $P_k:=\{n^k:n\in\N\}$ is enumerated by a linear recurrence with characteristic polynomial $(x-1)^{k+1}$. By the Hilbert-Waring Theorem, $(\Z,+,P_k)$ defines the ordering, and it also defines multiplication when $k\geq 2$ (see \cite[Proposition 6]{Bes}).  Another  unstable example is the expansion of $(\Z,+)$ by $\{q^n+n:n\in\N\}$, for any fixed integer $q\geq 2$, which is enumerated by a recurrence relation with characteristic polynomial $(x-q)(x-1)^2$ (see \cite[Theorem 4.8]{CoGG}).  In this section, we separate the stable examples from the unstable ones using the observation that, in each unstable example above, $1$ is a repeated root of $p_A(x)$.

\begin{theorem}\label{thm:LRR}
Suppose $\cG=(G,+)$ is a weakly minimal subgroup of $(\Q^{\alg},+)$, and $A\seq G$ is enumerated by a linear recurrence relation such that no repeated root of the characteristic polynomial  is a root of unity. Then, for any finite $F\seq G$ and any $B\seq A+F$, $(\cG,B)$ has nfcp and is superstable of $U$-rank at most $\omega$.
\end{theorem}

Note that this is a significant generalization of the previous results described above, since it includes any case where $p_A(x)$ is separable. We also note that the absence of roots of unity as repeated roots of $p_A(x)$ does not characterize stability of $(G,+,A)$ (see Remark \ref{rem:LRRstableunity} below). 

To prove Theorem \ref{thm:LRR}, we will use Lemma \ref{lem:ESSMA}, together with a number-theoretic tool of a similar flavor as Theorem \ref{thm:ESS}. To state this result, we need some further notation. For the rest of this section, let $A\seq \Q^{\alg}$ be enumerated by linear recurrence relation. We may assume $A$ is infinite. Fix a number field $K\seq\Q^{\alg}$ containing $\mu_1,\ldots,\mu_{d_*}$ and the coefficients of $\alpha_1(x),\ldots,\alpha_{d_*}(x)$. Given an integer $k\geq 1$ and a tuple $\bar{\lambda}=(\lambda_1,\ldots,\lambda_k)\in (K^*)^k$, define the function $\bar{\lambda}^{\xbar}\colon\Z^k\to K$ such that $\bar{\lambda}^{(n_1,\ldots,n_k)}=\lambda_1^{n_1}\cdot\ldots\cdot \lambda_k^{n_k}$. The following result (which holds for any number field) is a quantitative version of work of Laurent \cite{MLaur1,MLaur2}, due to Schlickewei and Schmidt \cite{SchSchPEE} (see also \cite[Theorem 12.1]{SchLRS}).

\begin{theorem}[Schlickewei \& Schmidt]\label{thm:QLT}
Fix $k,m\geq 1$ and, for each $i\in [m]$, fix $\bar{\lambda}_i\in (K^*)^k$ and $P_i(x_1,\ldots,x_k)\in K[x_1,\ldots,x_k]$ of degree $\delta_i$. Assume:
\begin{enumerate}[$(i)$]
\item no $P_i(\xbar)$ is identically $0$, and 
\item for any $\nbar\in\Z^k$, if $\bar{\lambda}_1^{\nbar}=\ldots=\bar{\lambda}_m^{\nbar}$ then $\nbar=\bar{0}$.
\end{enumerate}
Then there are $O_{K,m,k,\delta_1,\ldots,\delta_m}(1)$ tuples $\nbar\in\Z^k$ such that $\sum_{i=1}^mP_i(\nbar)\bar{\lambda}_i^{\nbar}=0$ and $\sum_{i\in I}P_i(\nbar)\bar{\lambda}_i^{\nbar}\neq 0$ for any nonempty $I\subsetneq [m]$.
\end{theorem}

Given $i\in[d_*]$, will  use the notation $\alpha^*_{\mu_i}(x)$ for $\alpha_i(x)$.  Set $\Lambda=\{\mu_1,\ldots,\mu_{d_*}\}$, and partition $\Lambda=\Lambda_0\cup\Lambda_1$ so that $\mu_i\in\Lambda_1$ if and only if $\mu_i$ is a root of unity. Let $\Phi$ denote the set of functions from $\N$ to $K$ of the form $x\mapsto c\alpha^*_\lambda(x)\lambda^x$ for some $\lambda\in\Lambda_0$ and $c\in\{1,\nv 1\}$. Let $\cK=(K,+)$ be the additive group in $K$.

\begin{lemma}\label{lem:LRRESSMA}
$\N^\Phi_{\cK}$ is mutually algebraic.
\end{lemma}
\begin{proof}
We will apply Lemma \ref{lem:ESSMA}. In particular, we show that for any $k\geq 1$, there is some $w_k\in\N$ such that $|\N_{0}(\bar{\varphi};r)|\leq w_k$ for any $\bar{\varphi}\in\Phi^k$ and $r\in K^*$. In particular, let $\delta$ be the maximum degree of any $\alpha_i(x)$, for $i\in [d_*]$. Given $k\geq 1$, let $w_k\in\N$ be greater than the $O_{K,k+1,k,\delta_1,\ldots,\delta_{k+1}}(1)$ bound from Theorem \ref{thm:QLT}, for any $\delta_1,\ldots,\delta_{k+1}\leq\delta$. 

Fix $k\geq 1$, $\bar{\varphi}\in\Phi^k$, and $r\in K^*$. For $i\in [k]$, let $c_i\in\{1,\nv 1\}$ and $\lambda_i\in\Lambda_0$ be such that $\varphi_i(x)=c_i\alpha^*_{\lambda_i}(x)\lambda_i^n$, and let $P_i(\xbar)\in K[x_1,\ldots,x_k]$ be the polynomial $c_i\alpha_{\lambda_i}(x_i)$. Let $P_{k+1}(\xbar)=\nv r$. For $i\in[k]$, let $\bar{\lambda}_i=(1,\stackrel{i-1}{\ldots}\,,1,\lambda_i,1,\stackrel{k-i}{\ldots}\,,1)\in (K^*)^k$. Let $\bar{\lambda}_{k+1}=(1,\stackrel{k}{\ldots}\,,1)$. Note that for any $\nbar\in\Z^k$, $\bar{\lambda}_{k+1}^{\nbar}=1$ and $\bar{\lambda}_i^{\nbar}=\lambda_i^{n_i}$ for any $i\in[k]$. In particular, $\N_0(\bar{\varphi};r)$ is precisely the set of solutions to $\sum_{i=1}^m P_i(\xbar)\bar{\lambda}_i^{\xbar}=0$ in $\N^k$ such that $\sum_{i\in I}P_i(\xbar)\bar{\lambda}_i^{\xbar}\neq 0$ for all nonempty $I\subsetneq[k]$.

Suppose $\nbar\in \Z^k$ is such that $\bar{\lambda}_1^{\nbar}=\ldots=\bar{\lambda}_{k+1}^{\nbar}$. Then $\lambda_1^{n_1}=\ldots=\lambda_k^{n_k}=1$, and so $n_i=0$ for all $i\in[k]$ since $\lambda_i$ is not a root of unity. Altogether, by Theorem \ref{thm:QLT}, we have $|\N_0(\bar{\varphi};r)|\leq w_k$. 
\end{proof}

We now assume that no $\lambda\in\Lambda_1$ is a repeated root of $p_A(x)$, and so $\alpha^*_{\lambda}(x)$ is a constant $\alpha^*_{\lambda}\in K^*$. 
Define
\[
B=\left\{\sum_{\lambda\in\Lambda_0}\alpha^*_{\lambda}(n)\lambda^n:n\in\N\right\}\makebox[.5in]{and}F=\left\{\sum_{\lambda\in \Lambda_1}\alpha^*_{\lambda}\lambda^n:n\in\N\right\}.
\]
Note that $B,F\seq K$ and $A\seq B+F$. Moreover, if $\lambda\in\Lambda$ then  $\{\lambda^n:n\in\N\}$ is finite if and only if $\lambda\in\Lambda_1$. So $F$ is finite and $B$ is infinite.

\begin{lemma}\label{lem:BMA}
$B^{\qf}_{\cK}$ is mutually algebraic.
\end{lemma}
\begin{proof}
Let $\Lambda_0=\{\lambda_1,\ldots,\lambda_{\ell}\}$ with $\ell\leq d_*$.
For $k\geq 1$, $\cbar\in\{1,\nv 1\}^k$, and $r\in K$, set
\[
D_{\cbar;r}=\left\{\nbar\in \N^k:\sum_{t=1}^k\sum_{i=1}^{\ell}c_t\alpha^*_{\lambda_i}(n_{t})\lambda_i^{n_{t}}=r\right\}.
\]
Then $D_{\cbar;r}$ is $\emptyset$-definable in $\N^{\Phi}_{\cK}$ since $\nbar\in D_{\cbar;r}$ if and only if, setting   
\[
\nbar^t=(n_t,\stackrel{\ell}{\ldots}\,,n_t)\makebox[.5in]{and}\bar{\varphi}_t=(c_t\alpha^*_{\lambda_1}(x)\lambda_1^x,\ldots,c_t\alpha^*_{\lambda_\ell}(x)\lambda_{\ell}^x)
\]
for $t\in [k]$, we have $(\nbar^1,\ldots,\nbar^k)\in A((\bar{\varphi}_1,\ldots,\bar{\varphi}_k);r)$. Let $E$ be the equivalence relation on $\N$ such that $E(m,n)$ holds if and only if
\[
\sum_{i=1}^{\ell}\alpha^*_{\lambda_i}(m)\lambda_i^{m}=\sum_{i=1}^{\ell}\alpha^*_{\lambda_i}(n)\lambda_i^{n}.
\]
Then $E$ is defined by $D_{(1,\nv1);0}\seq \N^2$, and thus is $\emptyset$-definable in $\N^{\Phi}_{\cK}$. Note also that, for any $\cbar\in\{1,\nv 1\}^k$ and $r\in K$, $D_{\cbar;r}$ is $E$-invariant as a subset of $\N^k$. 

Now $B^{\qf}_{\cK}$ is clearly interdefinable with the structure with universe $\N/E$ and relations $D_{\cbar;r}/E$ for all $k\geq 1$, $\cbar\in\{1,\nv 1\}^k$, and $r\in K$. So $B^{\qf}_{\cK}$ is mutually algebraic by Lemma \ref{lem:LRRESSMA} and Corollary \ref{cor:MAinter}. 
\end{proof}

\begin{corollary}\label{cor:AMA}
$A^{\qf}_{\cK}$ is mutually algebraic.
\end{corollary}
\begin{proof}
Let $\cM=(B+F)^{\qf}_{\cK}$. Then $A_{\cK}^{\qf}$ is a reduct of $A_{\cM}$, and so, as in the proof of Theorem \ref{thm:MAG-main}, it suffices to show that $\cM$ is mutually algebraic. Fix a finite set $F_0\seq B$ with $|F|=|F_0|$, and let $\sigma\colon F_0\to F$ be a bijection. Let $D=B\times F_0$, and note that $D\seq B^2$ is $B^{\qf}_{\cK}$-definable of $U$-rank $1$. Given $k\geq 1$, $\cbar\in\{1,\nv 1\}^k$, and $r\in K$, define
\[
D_{\cbar;r}=\left\{((b_1,f_1),\ldots,(b_k,f_k))\in D^k:\sum_{i=1}^k c_i(b_i+\sigma(f_i))=r\right\}.
\]
Then, for any $k\geq 1$, $\cbar\in\{1,\nv 1\}^k$, and $r\in K$, we have
\[
D_{\cbar;r}=\bigcup_{\bar{f}\in F_0^k}\left\{((b_1,f_1),\ldots,(b_k,f_k)):\sum_{i=1}^k c_ib_i=r-\sum_{i=1}^kc_i\sigma(f_i)\right\},
\]
and so $D_{\cbar;r}$ is $B^{\qf}_{\cK}$-definable. Moreover, the equivalence relation $E$ on $D$ given by $b_1+\sigma(f_1)=b_2+\sigma(f_2)$ is $B^{\qf}_{\cK}$-definable by $D_{(1,1,\nv 1,\nv 1);0}$, and any $D(\cbar;r)$ is $E$-invariant. Finally, $\cM$ is clearly interdefinable with the structure with universe $D/E$ and relations for $D(\cbar;r)/E$, for any $k\geq 1$, $\cbar\in\{1,\nv 1\}^k$, and $r\in K$. So $\cM$ is mutually algebraic by Lemma \ref{lem:BMA} and Corollary \ref{cor:MAinter}.
\end{proof}

As before, Corollary \ref{cor:AMA}, Theorem \ref{thm:MAG-main}, and Remark \ref{rem:subgroup} yield Theorem \ref{thm:LRR}.

\begin{remark}\label{rem:LRRSS}
Suppose $(G,+)$ is a weakly minimal subgroup of $(\Q^{\alg},+)$, $A\seq \Q^{\alg}$ is enumerated by a  linear recurrence relation, and no repeated root of $p_A(x)$ is a root of unity. Via Remark \ref{rem:SS} and Corollary \ref{cor:interpretG}$(b)$, Corollary \ref{cor:AMA} implies that $\Sigma_n^{\pm}(A)$ has upper Banach density $0$ (in $G$) for all $n\geq 1$.  In fact, if $p_A(x)$ is separable then one can use Theorem \ref{thm:ESS} to show that  for any $n\geq 1$, $\Sigma^{\pm}_n(A)$ does not contain arbitrarily large finite arithmetic progressions (see \cite[Remark 3.6]{CoGG}). 
\end{remark}

\begin{remark}\label{rem:LRRstableunity}
A root of unity appearing as a repeated root of $p_A(x)$ does not necessarily mean $(G,+,A)$ is unstable. For example, $\Z=\text{LRR}(2,0,\nv 1,0;0,0,1,\nv 1)$, which has characteristic polynomial $(x-1)^2(x+1)^2$. This situation  would likely be clarified by focusing on recurrence relations which are \emph{non-degenerate}, i.e., there do not exist distinct roots $\mu_i$ and $\mu_j$ of $p_A(x)$ such that $\mu_i/\mu_j$ is a root of unity. In general, any recurrence relation  can be  effectively partitioned into finitely many non-degenerate pieces (see \cite[Theorem 1.2]{EPSWbook}). Note also that if $A$ is non-degenerate and some root $\mu$ of $p_A(x)$ is a root of unity, then $\mu$ is the unique such root and $\mu\in\{1,\nv 1\}$.
A tentative conjecture is that if $A\seq \Z$ is enumerated by a linear recurrence relation as above, and some repeated root of $p_A(x)$ is a root of unity, then either $(\Z,+,A)$ is unstable or $A$ is degenerate.
\end{remark}

Finally, we point out that the restriction to $\Q^{\alg}$ in Theorem \ref{thm:LRR} is due to the fact that the work of Schlickewei and Schmidt from \cite{SchSchPEE} (namely, Theorem \ref{thm:QLT} above) applies only to number fields. Suppose instead that we have a set $A$ enumerated by a recurrence relation as above, but with $a_0,\ldots,a_{d-1},\beta_1,\ldots,\beta_d$ in an arbitrary algebraically closed field $\K$ of characteristic $0$. In order to carry out the work in this section, one would need a version of Theorem \ref{thm:QLT}, where $O_{K,m,k,\delta_1,\ldots,\delta_m}(1)$ is replaced by some bound depending only on $k$, $m$, and $A$. Such a result is known to hold in the case that $p_A(x)$ is separable, due to various ``specialization" techniques (see \cite{SchLRS}). On the other hand, we can use Theorem \ref{thm:ESS}, and arguments similar to the proof of Theorem \ref{thm:FGM}, to give a more direct argument. 

\begin{theorem}\label{thm:LRRC}
Let $\K$ be an algebraically closed field of characteristic $0$, and let $\cG=(G,+)$ be a weakly minimal subgroup of the additive group $\cK$  of $\K$. Fix $A\seq G$ enumerated by a linear homogeneous recurrence relation with constant coefficients in $\K$ and separable characteristic polynomial. Then, for any finite $F\subset G$ and any $B\seq A+F$, $(\cG,B)$ has nfcp and is superstable of $U$-rank at most $\omega$.
\end{theorem}
\begin{proof}
We use the same notation for $A$ as above, but with $\Q^{\alg}$ replaced by $\K$. Since $p_A(x)$ is separable, we have $d_*=d$. Moreover, for all $i\in[d]$, $\alpha_i(x)$ is a constant $\alpha_i\in\K^*$, which we also denote by $\alpha^*_{\mu_i}$. Let $\Lambda=\{\mu_1,\ldots,\mu_d\}$, and partition $\Lambda=\Lambda_0\cup\Lambda_1$ as above. Let $\Phi$ denote the set of functions from $\N$ to $\K$ of the form $x\mapsto c\alpha^*_\lambda\lambda^x$ for some $\lambda\in\Lambda_0$ and $c\in\{1,\nv 1\}$.  If we can show that $\N^{\Phi}_{\cK}$ is mutually algebraic, then the rest of the proof follows as above.

To show that $\N^{\Phi}_{\cK}$ is mutually algebraic, we fix $k\geq 1$, $\bar{\varphi}\in\Phi^k$, and $r\in\K$, and show that $\N_0(\bar{\varphi};r)$ is mutually algebraic as a $k$-ary relation on $\N$. Let $\bar{\varphi}=(\varphi_1,\ldots,\varphi_k)$ where $\varphi_i\colon x\mapsto c_i\alpha^*_{\lambda_i}\lambda_i^x$ for $\lambda_i\in\Lambda_0$ and $c_i\in\{1,\nv 1\}$. Let $\Gamma$ be the subgroup of $\K^*$ generated by $\lambda_1,\ldots,\lambda_k$, and let $\Delta$ be the set of $\xbar\in \Gamma^k$ such that $\sum_{i=1}^kc_i\alpha^*_{\lambda_i}x_i=r$ and $\sum_{i\in I}c_i\alpha^*_{\lambda_i}x_i\neq 0$ for all nonempty $I\subsetneq[k]$. Then the map $\sigma\colon \nbar\to (\lambda_1^{n_1},\ldots,\lambda_k^{n_k})$ is well-defined from $\N_0(\bar{\varphi};r)$ to $\Delta$, and  injective since no $\lambda_i$ is a root of unity. So it suffices to show $\Delta$ is mutually algebraic as a $k$-ary relation on $\Gamma$. This follows from Theorem \ref{thm:ESS} exactly as in the proof of Theorem \ref{thm:FGM}.  
\end{proof}

\begin{remark}
A recurrence sequence $(a_n)_{n=0}^\infty$ as above can be extended to to $(a_n)_{n\in\Z}$ using the same recurrence relation, and the representation of $a_n$ using the roots of $p_A(x)$ still holds. Thus the analogues of Theorems \ref{thm:LRR} and \ref{thm:LRRC} hold for a set $A\seq G$ enumerated in this fashion as well. In the proofs one only needs to replace $\N^{\Phi}_{\cK}$ by $\Z^{\Phi}_{\cK}$, where the maps in $\Phi$ are extended to $\Z$ in the obvious way.
\end{remark}

\bibliographystyle{amsplain}
\end{document}